\documentclass[a4paper,10pt,fleqn]{amsart}
\pdfoutput=1

\usepackage[utf8]{inputenc}
\usepackage[T1]{fontenc}

\usepackage{mathabx}
\def\VERT{\vvvert}

\usepackage{todonotes}

\usepackage{microtype}

\usepackage[english]{babel}
\usepackage{csquotes} 

\usepackage{setspace}
\usepackage{amsmath,amsthm,amsfonts,amssymb}
\usepackage{mathtools}

\usepackage[shortlabels]{enumitem}

\usepackage{xcolor}

\usepackage{xspace}

\usepackage[hypertexnames=false]{hyperref}

\definecolor{darkblue}{rgb}{0,0,0.5}
\definecolor{cerule}{RGB}{53,122,183}
\definecolor{cardinal}{RGB}{184,32,16}

\hypersetup{
    colorlinks,
    linkcolor={black},
    citecolor={cerule},
    urlcolor={cardinal}
}

\usepackage{float}
\usepackage[]{algpseudocode}
\algrenewcommand\textproc{}

\usepackage{tikz}
\usetikzlibrary{decorations,arrows,automata,trees,fit,shapes,intersections,positioning}
\definecolor{newLightBrown}{RGB}{230,51,18}

\usepackage[ruled]{caption}      
\usepackage{subcaption}

\floatstyle{ruled}
\newfloat{algo}{tp}{lop}
\floatname{algo}{Algorithm}
\allowdisplaybreaks[2]

\usepackage[
  citestyle=authoryear-comp,
  bibstyle=authoryear,
  backend=biber,
  isbn=false,
  maxcitenames=2,
  maxbibnames=6,
  firstinits=true,
  sortcites=true,
  url=false
]{biblatex}
\bibliography{maths}

\AtEveryBibitem{\clearfield{month}}
\AtEveryBibitem{\clearfield{day}}

\DeclareNameAlias{sortname}{first-last}

\DeclareFieldFormat{eprint:tel}{%
  \textsc{Tel}\addcolon\space\ifhyperref
    {\href{https://hal.archives-ouvertes.fr/tel-#1}{\nolinkurl{#1}}}
    {\nolinkurl{#1}}}

\def\Id{\operatorname{id}}
\def\one{\mathbf{1}}

\def\bC{{\mathbb{C}}}
\def\bE{{\mathbb{E}}}

\def\bP{{\mathbb{P}}}

\def\bR{{\mathbb{R}}}
\def\bS{{\mathbb{S}}}

\def\cH{{\mathcal{H}}}
\def\cL{{\mathcal{L}}}

\def\cN{{\mathcal{N}}}

\def\cU{{\mathcal{U}}}
\def\cV{{\mathcal{V}}}
\def\cX{{\mathcal{X}}}

\def\bfu{{\mathbf u}}
\def\bfw{{\mathbf w}}
\def\bfv{{\mathbf v}}
\def\bfh{{\mathbf h}}
\def\bfy{{\mathbf y}}

\def\Frob{{\mathrm{Frob}}}

\def\fkg{{\mathfrak g}}

\def\Proj{\bP^n}

\def\mop{\operatorname}
\newcommand\smashintlong[2][]{\int_{\crampedrlap{#2}}\mathrlap{#1} {}_{\phantom{#2}}}
\newcommand\smashint[2][]{\int_{\crampedrlap{#2}}{#1} }

\def\leq{\leqslant}
\def\geq{\geqslant}

\def\det{\mop{det}\nolimits}
\def\codim{\mop{codim}}
\def\epsilon{\varepsilon}
\def\phi{\varphi}

\def\tdert{\tfrac{\ud}{\ud t}}

\def\tn#1{\left\VERT #1 \right\VERT}

\def\rstd{\rho_{\mathrm{std}}}

\newcommand{\abs}[1]{\left| #1 \right|}

\def\ud{\mathrm{d}}

\newcommand{\diffblock}[1]{#1}

\def\eqdef{\stackrel{.}{=}}
\newcommand{\st}{\mathrel{}\middle|\mathrel{}}

\newtheorem{theorem}{Theorem}
\newtheorem{proposition}[theorem]{Proposition}
\newtheorem{lemma}[theorem]{Lemma}
\newtheorem{corollary}[theorem]{Corollary}

\theoremstyle{definition}

\theoremstyle{remark}

\newtheorem{remark}[theorem]{Remark}

\title[Rigid continuation paths]
{Rigid continuation paths \\
I. Quasilinear average complexity\\for solving polynomial systems}
\author[P. Lairez]{Pierre Lairez}
\address{
Inria Saclay Île-de-France,
91120 Palaiseau,
France}

\date{\today}

\subjclass[2000]{Primary 68Q25; Secondary 65H10, 65H20, 65Y20}

\begin{document}

\begin{abstract}
  How many operations do we need on average to compute an approximate root of a random Gaussian polynomial system? Beyond Smale's 17th problem that asked whether a polynomial bound is possible, we prove a quasi-optimal bound $\text{(input size)}^{1+o(1)}$. This improves upon the previously known $\text{(input size)}^{\frac32 +o(1)}$ bound.

  The new algorithm relies on numerical continuation along \emph{rigid continuation paths}. The central idea is to consider rigid motions of the equations rather than line segments in the linear space of all polynomial systems. This leads to a better average condition number and allows for bigger steps. We show that on average, we can compute one approximate root of a random Gaussian polynomial system of~$n$ equations of degree at most~$D$ in~$n+1$ homogeneous variables with~$O(n^4 D^2)$ continuation steps. This is a decisive improvement over previous bounds that prove no better than~$\sqrt{2}^{\min(n, D)}$ continuation steps on average.
\end{abstract}

\maketitle
\tableofcontents
\section{Introduction}

Following a line of research opened in the 20th century by
\textcites{Smale_1985}{Smale_1986}{Renegar_1987,Renegar_1989}{Demmel_1988}{Shub_1993}{ShubSmale_1996,
  ShubSmale_1994,ShubSmale_1993a,ShubSmale_1993b,ShubSmale_1993,Malajovich_1994}
and developped in the 21st century by \diffblock{\textcite{BurgisserCucker_2013,
  BurgisserCucker_2011, BriquelCuckerPenaEtAl_2014, Malajovich_2018,
  ArmentanoBeltranBurgisserEtAl_2016, ArmentanoBeltranBurgisserEtAl_2018,
  BeltranPardo_2008, Beltran_2011, BeltranPardo_2011, BeltranPardo_2009,
  BeltranPardo_2009a, BeltranShub_2009, Lairez_2017, HauensteinSottile_2012,
  HauensteinLiddell_2016, BatesHauensteinSommeseEtAl_2013}}, to name a few,
I am interested in the number of elementary operations that one needs to compute
one zero of a polynomial system in a numerical setting. On this topic, Smale's
question is a landmark: ``Can a zero of~$n$ complex polynomial equations in~$n$
unknowns be found approximately, on average, in polynomial time with a
uniform algorithm?'' \parencite[17{th} problem]{Smale_1998}. The wording is
crafted to have a positive answer in spite of two major obstacles. The first one
is the NP-completeness of many problems related to deciding the feasability of a
polynomial system. Here, we consider well determined systems (as many equations
as unknowns),
over the complex numbers, in the average (\emph{a fortiori} generic) case, so
there will always be a zero. The second obstacle is the number of zeros: it is not polynomially
bounded in terms of the input size (the number of coefficients that define the
input system). Here, we ask for only one zero and numerical methods can take
advantage of it.

Smale's question is now solved \parencite{BeltranPardo_2009a,
  BurgisserCucker_2011, Lairez_2017}; it is an achievement but not an end. The
most obvious question that pops up is to improve the degree hidden behind the
words ``polynomial time''. This article presents an optimal answer, bringing
down ``polynomial time'', that is~$N^{O(1)}$, where~$N$ is the input size, to
``quasilinear time'', that is $N^{1+o(1)}$. The previous state of the art
was~$N^{\frac32 + o(1)}$ \parencite{ArmentanoBeltranBurgisserEtAl_2016}.

\subsection{State of the art}
\label{sec:state-art}

Let~$n$ and~$d_1,\dotsc,d_n$ be positive integers, and let~$\cH$ be the vector
space of tuples~$(f_1,\dotsc,f_n)$ of complex homogeneous polynomials of respective
degrees~$d_1,\dotsc,d_n$ in the variables~$x_0,\dotsc,x_n$.
Let~$D$ denote $\max(d_1,\dotsc,d_n)$.

We are interested in the average complexity of finding one zero of a polynomial
system, given as an element of~$\cH$. The complexity is measured with respect to the
\emph{input size}, denoted $N$. This is the number of complex coefficients that describe
a system, namely
\[ N \eqdef \dim_{\bC} \cH = \binom{d_1+n}{n} + \dotsb + \binom{d_n+n}{n}. \]
Note that~$N \geq 2^{\min(n, D)}$.
``{Average complexity}'' means that $\cH$ is endowed with a probability measure
(uniform on the unit sphere for some suitably chosen Hermitian norm) and that
the behaviour the algorithms is analyzed \emph{on average}, assuming that
the input is distributed according to this probability measure. We will
make use of randomized algorithms, that draw random numbers
during their execution. In this case, the average complexity is an average with
respect to both the input's distribution and the randomness used internally by the algorithm.

\subsubsection{Classical theory}

In the Shub--Smale--Beltrán--Pardo--Bürgisser--Cucker way of doing things, we
compute a zero of a homogeneous polynomial system~$F \in \cH$ by numerical
continuation from a random system~$G\in \cH$ of which we happen to know a zero
$\zeta \in \bP(\bC^{n+1})$. The continuation is performed along the
deformation~$F_t \eqdef \frac{1}{\|tF+(1-t)G\|}(t F + (1-t)G)$. Starting
from~$t=0$, we repeatedly increment the parameter~$t$ and track a zero of~$F_t$
with a projective Newton iteration applied to the previous approximation of the
zero. If the increment is small enough then we can be sure not to loose the zero
and to obtain, when~$t$ reaches~$1$, an approximate zero of the target
system~$F$. The total complexity of the algorithm depends on the number of
continuation steps that are performed, which in turn depends on the size of the
increment. The key issue is to specify how small is ``small enough''.

\textcite*{Smale_1986} gave a sufficient condition for the Newton iteration to
converge in terms of the \emph{gamma number} $\gamma(F, z)$, depending on a
polynomial system~$F$ and a point~$z$
(see \S \ref{sec:splitgamma}\eqref{eq:24} for a definition). 
Difficulties in estimating the variations of~$\gamma(F, z)$ with respect to~$F$
led \textcite{ShubSmale_1993b} to consider the \emph{condition number}
$\mu(F,z)$, which upperbounds~$\gamma(F,z)$ and characterizes how much a
zero~$z$ of a system~$F$ is affected by a small pertubation of~$F$.
They gave a sufficient condition for a continuation step to be small enough
in terms of $\mu$.
After some
refinements, \textcite{Shub_2009} proved that~$K(F,G,\zeta)$, the minimal number
of steps to go from~$G$ to~$F$ while tracking the zero~$\zeta$, is bounded by
\begin{equation}\label{eq:1}
K(F,G,\zeta) \leq \text{(constant)} \int_0^1 \mu(F_t,\zeta_t) \sqrt{\|\dot
    F_t\|^2 + \|\dot \zeta_t\|^2} \ud t.
\end{equation}
This is called the ``$\mu$ estimate''. Explicit algorithms that achieve this
bound have been designed by \textcite{Beltran_2011,DedieuMalajovichShub_2013,HauensteinLiddell_2016}.
A simpler but weaker form, called the ``$\mu^2$ estimate'', reads
\begin{equation}\label{eq:2}
K(F,G,\zeta) \leq \text{(constant)} D^\frac32  d_\bS(F,G) \int_0^1 \mu(F_t,\zeta_t)^2 \ud t,
\end{equation}
where~$d_\bS(F,G)$ is the distance in the unit sphere $\bS(\cH)$ from~$F$ to~$G$, that is the
length of the continuation path. It is often used in practice because it is much
easier to design algorithms that achieve this bound rather than the former. In
one form or the other, this kind of integral estimate for the number of steps is
the first mainstay of the method.

The second mainstay is a procedure discovered by \textcites{BeltranPardo_2011}
and simplified by \textcite{BurgisserCucker_2011} to sample a Gaussian random
system~$G\in\cH$ together with one of its zeros without the need for solving a
polynomial system: (1)~sample a random Gaussian linear map~$L : \bC^{n+1} \to
\bC^n$, (2)~compute a nonzero vector~$\zeta \in \bC^{n+1}$ in the kernel of~$L$
and (3)~sample a random Gaussian system in the affine subspace of~$\cH$ of all
systems~$G$ such that~$G(\zeta) = 0$ and~$\ud_\zeta G = L$. By construction, we
obtain a system~$G$ and one of its zeros~$\zeta$. Less trivially, $G/\|G\|$ is
uniformly distributed in the sphere~$\bS(\cH)$. We could think of a simpler
procedure that (1)~samples some~$\zeta \in \bC^{n+1}$ isotropically and
(2)~samples a random Gaussian system in the linear subspace of~$\cH$ of all
systems~$G$ such that~$G(\zeta) = 0$. This also gives a random system with one
of its zeros, by construction, but the system is not uniformly distributed in the
sphere after normalization.

These two mainstays together give a randomized algorithm to compute a zero of a
polynomial system and a way to analyze its average complexity on a random input.
On input $F\in \bS(\cH)$, the algorithm is: (1)~uniformly sample a random
system~$G\in \bS(\cH)$ together with a zero~$\zeta$, (2)~perform the numerical
continuation from~$G$ to~$F$ tracking the zero~$\zeta$. If~$F$ itself is a
uniformly distributed random variable, then for any~$t\in[0,1]$, $F_t$ is also uniformly
distributed, so~$(F_t,\zeta_t)$ has the same distribution as~$(G,\zeta)$.
Therefore, the average number of steps performed by the algorithm is bounded by
\begin{align*}
  \bE\left[K(F, G, \zeta)\right] &\leq \text{(constant)} \, \bE \left[  D^\frac32 d_\bS(F,G) \int_0^1 \mu(F_t,\zeta_t)^2 \ud t \right]\\
                      &\leq \text{(constant)}  D^\frac32 \int_0^1 \bE\left[\mu(F_t,\zeta_t)^2\right] \ud t\\
                      &\leq \text{(constant)}  D^\frac32 \, \bE \left[ \mu(G,\zeta)^2 \right].
\end{align*}

This leads us to the third mainstay: estimates for~$\bE \left[ \mu(G,\zeta)^2
\right]$. \textcite[Theorem~23]{BeltranPardo_2011} proved that~$\bE \left[
  \mu(G,\zeta)^2 \right] \leq nN$. Therefore, the average number of steps
performed by the algorithm on a random input is
\[ \bE[K(F, G, \zeta)] \leq \text{(constant)} nN. \]
The cost of each continuation step (basically, the computation of~$\mu$ and a Newton's
iteration) can be done in~$O(N)$ operations (when~$D \geq 2$). All in all, the
total average complexity of the classical algorithm is~$O(n D^\frac32 N^2)$ as~$N\to
\infty$. When $\min(n,D)\to \infty$, then this is~$N^{2+o(1)}$.

\subsubsection{Improvements}

How can we improve upon this complexity bound? We cannot do much about the
$O(N)$ cost of a continuation step, as it is already optimal. Concerning the number
of steps, we can try to use the $\mu$ estimate instead of the $\mu^2$ estimate.
Bounding~$\|\dot \zeta_t\|$ by~$\mu(F_t,\zeta_t)\|\dot F_t\|$ (which turns the
$\mu$~estimate into the $\mu^2$~estimate) is optimal in the worst case, but on
average, when the direction~$\dot F_t$ is random, this is pessimistic.
Building upon this idea, \textcite{ArmentanoBeltranBurgisserEtAl_2016} proved
that~$O(nD^\frac32 N^\frac12)$ continuation steps are enough on average.
This leads to a total average complexity of~$N^{\frac32 + o(1)}$ operations.

\textcite{BeltranShub_2009} proved that there exist continuation paths that
makes the $\mu$ estimate polynomially bounded in terms of~$n$ and~$D$. The
construction is explicit but it requires the knowledge of a zero of a target
system. This prevents it from being used algorithmically. Yet, it was the first
time that the possibility of performing numerical continuation in very few steps
(polynomially many with respect to~$n$ and~$D$, not~$N$) was supported.

Lastly,
let us mention that \textcite{HauensteinLiddell_2016} developped a $\gamma$
estimate,
based on Smale's $\gamma$ number.
It may be used as a starting point to obtain
very similar results to ours in a more traditional context. However, this
direction is yet to be explored.

\subsection{Contribution}

I describe a randomized algorithm that uses a numerical continuation from a
random start system to find one root of the input system. It performs $O(n^4 D^2)$
steps on average (and each step costs~$O(n^2D^2N)$ operations) for random
Gaussian systems. This
leads to a total average complexity of~$O(n^6 D^4 N)$ operations as~$N\to
\infty$ to find one approximate root of a random Gaussian system (Theorem~\ref{thm:main-result}). When
$\min(n,D)\to \infty$, this is~$N^{1+o(1)}$.
The algorithm relies on analogues
of the three mainstays of the classical theory: integral estimate for the number
of steps, randomization of the start system and average analysis of some
condition number. However, the basic tools are thoroughly renewed.

The starting point is the observation that a typical system in~$\cH$ is
poorly conditioned. As mentionned above, the expected squared condition number of a
random system at a random zero is bounded by~$nN$ and it turns out that this is
rather sharp. In view of Smale's question, this is satisfying, much more than
bounds involving the total number of zeros, but this~$N$ is the limit of the
method.

To improve the average conditioning, an idea is to define the notion of
conditioning with respect to a much lower dimensional parameter space, but still
big enough to be able to develop an analogue of Beltrán and Pardo's algorithm. I
propose here the \emph{rigid} setting, where the parameter space is not the
whole space of polynomial systems, but the group~$\cU$ made of~$n$ copies of the unitary
group~$U(n+1)$, of real dimension~$\sim n^3$. It acts by rigid motions on
the $n$~components of a fixed, well determined, polynomial system.
Figure~\ref{fig:rigid-path} illustrates a \emph{rigid continuation path.}

Less parameters is less opportunities for a dramatic pertubation that will ruin
the conditioning of a system. Beyond that, the continuation paths in the rigid
setting preserve the geometry of the input equations. This opens a way for
studying the average complexity of solving certain \emph{structured} systems.
Forthcoming work will address this topic.

A noteworthy contribution is the introduction of the \emph{split
  gamma number} which tightly upper bounds Smale's gamma number
(Theorem~\ref{thm:gamma-leq-sgamma})
and which allows for interesting average analysis, see~\S \ref{sec:gauss-rand-syst}.
Finally, the foremost outcome of the rigid setting is
Theorem~\ref{thm:complexity-uncouple}
which gives an average bound on the necessary number of continuation steps to
compute one root of a random system, with only a unitary invariance hypothesis
on the probability distribution.

\begin{figure}[t]
  \centering
  \begin{minipage}[t]{.3\linewidth}
    \begin{tikzpicture}[]
      \draw[thick] (0,0) circle (1); \path[fill=newLightBrown] (-.6,.8) circle
      (.15em);
      \begin{scope}
        \clip (-1.5,-1.5) rectangle (1.5,1.5); \draw[name path=L,thick] (-2,1)
        -- (2, 0); \path[fill=newLightBrown] (.666,.3333) circle (.15em);
      \end{scope}
    \end{tikzpicture}
    \subcaption{Compute one solution of each equation.}
  \end{minipage}\hfill
  \begin{minipage}[t]{.3\linewidth}
    \begin{tikzpicture}[] \draw[name path=C,thick] (0,0) circle (1);

      \begin{scope} \clip (-1.5,-1.5) rectangle (1.5,1.5); \foreach \k in
        {4,...,4} { \draw[name path=L,thick] (-2,-\k+1) -- (2, \k+\k*\k/4);
          \path[name intersections={of=C and L}];
          \path[fill=newLightBrown] (intersection-1) circle (.15em); }
      \end{scope}
    \end{tikzpicture}
    \subcaption{Move the hypersurfaces to make the solution match.}
    \end{minipage}\hfill
    \begin{minipage}[t]{.3\linewidth}
      \begin{tikzpicture}[] \draw[name path=C,thick] (0,0) circle (1);
        \begin{scope} \clip (-1.5,-1.5) rectangle (1.5,1.5); \foreach \k in
          {0,...,4} { \draw[name path=L,opacity=1-.2*(\k),thick] (-2,-\k+1) --
            (2, \k+\k*\k/4); \path[name intersections={of=C and L}];
            \path[fill=newLightBrown,opacity=1-.2*(\k)] (intersection-1) circle
            (.15em); }
        \end{scope}
      \end{tikzpicture}
      \subcaption{Continuously return to the original position while tracking
        the solution.}
    \end{minipage}
  \caption{Resolution of a polynomial system with a rigid continuation path.
    \label{fig:rigid-path}}
\end{figure}
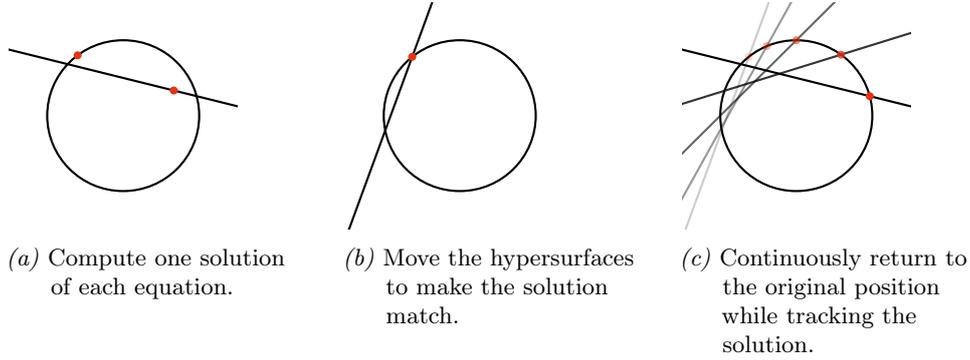

\subsubsection*{Ackowledgment}

It is my pleasure to thank Carlos Beltrán, Peter Bürgisser and Felipe Cucker for
many helpful discussions and valuable comments.
I am very grateful to the referees for their conscientious work.

\subsection{Notations and basic definitions}
\label{sec:notat-basic-defin}

{\raggedright

\begin{description}[align=left,labelwidth=6em,leftmargin=6em,style=nextline]
\item[$n$] some positive integer (used as the \emph{number of nonhomogeneous variables}).
\item[$\bP^n$] complex projective space of dimension~$n$.
\item[{$[z]$}] projective class of some nonzero~$z\in\bC^{n+1}$.
\item[$d_\bP$] geodesic distance on~$\Proj$ endowed with the Fubini-Study
  metric, that is $d_\bP([x], [y]) = \mop{arcsin} \left(
      \sqrt{1-| \langle x, y\rangle |^2} \right)$ for any $x, y \in \bS(\bC^{n+1})$.
\item[$H_d$] space of complex homogeneous polynomials of degree~$d$
  in~$x_0,\dotsc,x_n$.
\item[$r$] some positive integer (used as the \emph{number of equations}).
\item[$d_1,\dotsc,d_r$] some positive integers (used as the \emph{degrees of
    the equations}).
\item[$D$] the maximum of~$d_1,\dotsc,d_r$.
\item[{$\cH[r]$}] space of homogeneous systems of~$r$ equations of
  degree~$d_1,\dotsc,d_r$, that is~$H_{d_1}\times \dotsb \times H_{d_r}$.
  Elements of~$\cH[r]$ are often considered as polynomial maps~$\bC^{n+1}\to\bC^r$.
\item[$N$] the \emph{input size}, defined as~$\dim_\bC \cH[r]$.
\item[$U(k)$] group of unitary $k\times k$ matrices.
\item[$u^*$] conjugate transpose of~$u$.
\item[$\cU$] the group of $r$-uples of unitary matrices, $U(n+1)^r$. Elements
  of~$\cU$ are denoted in boldface, like~$\bfu$.
\item[$\one_\cU$] $(\Id,\dotsc,\Id)$, the neutral element in~$\cU$.
\item[$\| - \|$] norm in a Hermitian space.
\item[$\| - \|_W$] Weyl norm of a polynomial \parencite[see][\S16.1]{BurgisserCucker_2013}.
\item[$\tn{-}$] operator norm of a map between Hermitian spaces.
  For a multilinear map~$\phi : E^k \to V$, this is $\sup \left\{
    \|\phi(e_1,\dotsc,e_k)\| \st \|e_1\| = \dotsb = \|e_k\| = 1 \right\}$.
\item[$\| - \|_\Frob$] Frobenius norm of a map between Hermitian spaces.
\item[$\|-\|_u$] $1/\sqrt{2}$ times~$\|-\|_\Frob$ (used as the Riemannian metric
  on the tangent spaces of~$U(n+1)$).
\item[$\phi^\dagger$] Moore-Penrose pseudo-inverse of a surjective linear map~$\phi
  : E \to F$, it is the unique linear map  $\psi : F\to E$ such that~$\phi \psi = \Id_F$
  and~$\psi\phi$ is the orthogonal projection onto the row space of~$\phi$ (the
  orthogonal complement of the kernel).
\item[$\ud_z F$] the derivative of some polynomial map~$F : \bC^{n+1} \to \bC^r$
  at~$z\in\bC^{n+1}$. We will use the same notation with~$z\in \Proj$, which means
  that we choose a representative $\bar z \in \bC^{n+1}$ of~$z$ such
  that~$\|\bar z\|=1$.
\item[$\ud_z F^\dagger$] the pseudo-inverse of the derivative.
\item[$\cN_F$] projective Newton's operator associated to~$F$
\item[$\eqdef$] ``is defined as''
\item[$A = O(B)$ as $C\to \infty$]
  ``there are~$C_0 \geq 0$ and~$k \geq 0$ such that~$C \geq C_0 \Rightarrow A \leq kB$''.
\item[standard normal variable]
  a Gaussian random variable of an Euclidean space with unit covariance matrix
  in some orthonormal basis.
  The notion is relative to the underlying Euclidean inner product.
  For a Hermitian space, we consider the induced Euclidean structure.
\end{description}}

\section{Rigid solution varieties}
\label{sec:rigid-solut-vari}

The classical solution variety is the subvariety of~$\cH[r]\times \Proj$ of
all~$(F,\zeta)$ such that~$\zeta$ is a zero of~$F$. We now introduce an
analogue variety in the rigid setting. 

Let~$X_1,\dotsc,X_r$ be pure-dimensional subvarieties of $\Proj$, with~$\sum_i
\mop{codim} X_i \leq n$.
Let~$\cU$ denote the group~$U(n+1)^r$. It acts naturally on the
product~$(\Proj)^r$ of $r$~copies of the projective space. We denote its elements
in boldface~$\bfu = (u_1,\dotsc,u_r)$. Let~$\cX$ denote the product
variety~$X_1\times\dotsb\times X_r \subset (\Proj)^r$. For~$\bfu \in \cU$,
let~$\bfu \cX$ denote the image of~$\cX$ under the action of~$\bfu$, that
is~$\prod_{i=1}^r u_iX_i$, and let~$\cap \bfu\cX$ denote the intersection
$\cap_{i=1}^r u_i X_i \subseteq \Proj$. The \emph{rigid solution variety} is
defined as
\[ \cV \eqdef \left\{ (\bfu, x)\in \cU \times \Proj \st x \in \cap \bfu\cX
  \right\}.\]
There is not a single solution variety, but rather any choice of
subvarieties~$X_1,\dotsc,X_r$ leads to a solution variety.
In this section, we will study the geometry of~$\cV$ with~$X_1,\dotsc,X_r$
fixed. Later on, we will assume that $X_1,\dotsc,X_r$ are hyper\-surfaces
defined by random polynomials.

Let~$\mathbb{G}(k)$ denote the Grassmannian of $k$-dimensional projective subspaces
of~$\Proj$, that is $k+1$-dimensional linear subspaces of~$\bC^{n+1}$. For a smooth
point~$x \in X_i$, the projectivization of the tangent space of the cone over~$X_i$ at
some representative~$\bar x \in \bC^{n+1}$ of~$x$ is an element
of~$\mathbb{G}(k_i)$ and is denoted~$\mathbb{T}_x X_i$.
If~$X_i$ is the zero set of some homogeneous polynomial system~$F_i\in \cH[m]$,
then~$\mathbb{T}_x(u_i X_i)$ is the projectivization of the kernel of~$\ud_x(F_i\circ u_i^{-1})$.
Let~$\cL \eqdef \mathbb{G}(\dim X_1) \times \dotsb
\times \mathbb{G}(\dim X_r)$, and for $\bfh = (h_1,\dotsc,h_r)\in\cL$, let~$\cap
\bfh$ denote the intersection of the~$h_i$ in~$\Proj$. To a generic point $(\bfu, x)$
of~$\cV$, we associate the linearization
\[ L(\bfu,x) \eqdef \big( \mathbb{T}_x(u_1 X_1),\dotsc, \mathbb{T}_x(u_r X_r) \big)
  \in \cL.\]
Note that~$x\in \cap L(\bfu, x)$.

This section aims at three goals: describe precisely the so-called
\emph{standard} distribution on~$\cV$ (Theorem~\ref{thm:linearization}), give an algorithm to
sample from this distribution (Algorithm~\ref{algo:sample}) and
define the \emph{split gamma number}, a variant of the gamma number well adapted
to the rigid setting.

\subsection{Determinant of subspaces and incidence condition number}
\label{sec:orth-determ-incid}

Let~$E_1$,\ldots, $E_r$ be nonzero linear subspaces of a Hermitian space~$V$.
Let~$\pi_i$ be the orthogonal projector on~$E_i$.
We define the \emph{multiprojection} map $\mop{proj}(E_1,\dotsc,E_r)$ by
\begin{align*}
  \mop{proj}(E_1,\dotsc,E_r) : V &\longrightarrow E_1 \times \dotsb \times E_r \\
  v &\longmapsto ( \pi_{i}v, \dotsc, \pi_{r} v).
\end{align*}
We say that the family~$E_1,\dotsc,E_r$ is \emph{nondegenerate} if~$\sum_i \dim E_i =
\dim \left( \sum_i E_i \right)$, or, equivalently, when the multiprojection map
is surjective.



We define the \emph{determinant} of~$E_1,\dotsc,E_r$ as
\[ \det(E_1,\dotsc,E_r) \eqdef \abs{ \det \left( 
      \mop{proj}(E_1,\dotsc,E_r)_{| E_1 + \dotsb + E_r }
    \right)}, \]
in the nondegenerate case and~$\det(E_1,\dotsc,E_r) \eqdef 0$ otherwise.
Note that the determinant of a map between two Hermitian spaces is well defined
up to multiplication by some~$e^{i\theta}$, so that the modulus is well defined.
We also define the \emph{orthogonal determinant} of~$E_1,\dotsc,E_r$ as
\[ \det^\perp(E_1,\dotsc,E_r) \eqdef \det(E_1^\perp,\dotsc,E_r^\perp). \]
Lastly, we define the \emph{incidence condition number} of~$E_1,\dotsc,E_r$ as
\[ \kappa(E_1,\dotsc,E_r)\eqdef \tn{\mop{proj}(E_1,\dotsc,E_r)^\dagger} \]
when the multiprojection map is surjective, and~$\kappa(E_1,\dotsc,E_r) \eqdef
\infty$ otherwise.
With the appropriate distance, the incidence condition number is the inverse of
the distance of the tuple~$(E_1,\dotsc,E_r)$ to the closest~$(F_1,\dotsc,F_r)$
such that $\dim \left( \sum_i F_i \right) < \sum_i \dim F_i$
\parencite[Theorems~1.1 and~1.3]{BreidingVannieuwenhoven_2018}.


\begin{lemma}
  \label{lem:kappa-geq-1}
  For any subspaces~$E_1,\dotsc,E_r \subseteq V$,
  $\kappa(E_1,\dotsc,E_r) \geq 1$, with equality if and only if~$E_1,\dotsc,E_r$ are
  orthogonal subspaces.
\end{lemma}

\begin{proof}
  If the family~$E_1,\dotsc,E_r$ is degenerate
  then $\kappa(E_1,\dotsc,E_r)= \infty \geq 1$,
  so we may assume it is nondegenerate.
  The Hermitian transpose of~$\mop{proj}(E_1,\dotsc,E_r)$ is 
  the map~$\mop{sum}(E_{1},\dotsc,E_r) : E_1\times \dotsb\times E_r \to V$ defined by
  \begin{equation*}
    \mop{sum}(E_1,\dotsc,E_r) :   (u_1,\dotsc,u_r) \longmapsto u_1 + \dotsb + u_r,
  \end{equation*}
  so that~$\kappa(E_1,\dotsc,E_r) = \tn{\mop{proj}(E_1,\dotsc,E_r)^\dagger} = \tn{\mop{sum}(E_1,\dotsc,E_r)^\dagger}$.
  In the nondegenerate case, $\mop{sum}(E_1,\dotsc,E_r)$ is injective,
  and~$\mop{sum}(E_1,\dotsc,E_r)^\dagger$ is a left inverse, that is,
  for any~$(v_1,\dotsc,v_r) \in \prod_i E_i$,
  \begin{equation*}
    (v_1,\dotsc,v_r) = \mop{sum}(E_1,\dotsc,E_r)^\dagger (v_1 +\dotsb + v_r),
  \end{equation*}
  and, in particular,
  \begin{equation}\label{eq:49}
    \sum_i \|v_i\|^2 \leq \tn{ \mop{sum}(E_1,\dotsc,E_r)^\dagger }^2 \bigg\| \sum_i v_i \bigg\|^2.
  \end{equation}
  Choosing~$v_j \neq 0$ for some~$j$  and~$v_i = 0$ for~$i \neq j$ shows that
  $\tn{\mop{sum}(E_1,\dotsc,E_r)^\dagger } \geq 1$.

  If $\tn{\mop{sum}(E_1,\dotsc,E_r)^\dagger } = 1$,
  we have by~\eqref{eq:49}, for any~$v_i \in E_i$ and~$v_j\in E_j$, $i\neq j$,
  \begin{equation*}
    \|v_i\|^2+\|v_j\|^2 \leq \|v_i + v_j\|^2 = \|v_i\|^2 + \|v_j\|^2 + 2 \Re \langle v_i, v_j \rangle,
  \end{equation*}
  which implies that the real part of~$\langle v_i,v_j \rangle$ is nonnegative.
  Since it holds for~$-v_i$ and~$v_j$ too, it follows that~$E_i$ and~$E_j$
  are orthogonal.
  Conversely, if $E_1,\dotsc,E_r$ are orthogonal, then the map~$\mop{sum}(E_1,\dotsc,E_r)$
  is an isometric embedding, and thus~$\kappa(E_1,\dotsc,E_r) = 1$.
\end{proof}

\begin{lemma}
  \label{lem:multiproj-gram}
  Let~$P \eqdef \pi_1 +\dotsb+\pi_r$, it is a self-adjoint endomorphism of the
  subspace~$\sum_i E_i\subseteq V$.
  In the nondegenerate case,
  $\det(E_1,\dotsc,E_r)^2 = \det P$.
\end{lemma}

\begin{proof}
  Since $P = \mop{sum}(E_1,\dotsc,E_r) \cdot
  \mop{proj}(E_1,\dotsc,E_r)_{|E_1+\dotsb+E_r}$
  and since the map $\mop{sum}(E_1,\dotsc,E_r)$ is the Hermitian transpose
  of~$\mop{proj}(E_1,\dotsc,E_r)_{|E_1+\dotsb+E_r}$,
  it follows the definition of~$\det(E_1,\dotsc,E_r)$ that
  $\det P = \det(E_1,\dotsc,E_r)^2$.
\end{proof}

We can check similarly that $\kappa(E_1,\dotsc,E_r)^2 = \tn{P^{-1}}$, giving
another interpretation of~$\kappa$ (which will not be used here).

\begin{lemma}
  \label{lem:multiangle}
  For any subspaces~$E_1,\dotsc,E_r \subseteq V$,
  \[ \det^\perp(E_{1}, \dotsc, E_r) = \det^\perp(E_1, E_2)\det^\perp(E_1 \cap
    E_2, E_{3},\dotsc, E_r). \]
\end{lemma}

\begin{proof}
  This follows from the factorization
  \begin{multline*}
    \mop{proj}(E_1^\perp,\dotsc,E_r^\perp)_{\sum_i E_i^\perp} = \left(
      \mop{proj}(E_1^\perp, E_2^\perp)_{|E_1^\perp + E_2^\perp} \times
      \Id_{E_3^\perp} \times \dotsb \times \Id_ {E_r^\perp} \right)
    \circ \\
    \mop{proj}(E_1^\perp + E_2^\perp, E_3^\perp, \dotsc,E_r^\perp)_{| \sum_i E_i^\perp
    }. \qedhere
  \end{multline*}
\end{proof}

\subsection{Reminders on Riemannian geometry}
\label{sec:remind-riem-geom}

We will work mainly with two Riemannian manifolds: $\Proj$, the $n$-dimensional
complex projective space endowed with the Fubini--Study metric, and~$U(n+1)$,
the group of $(n+1)\times (n+1)$ unitary matrices. Concerning the latter, we
endow~$\bC^{(n+1) \times (n+1)}$ with the norm
\begin{equation}\label{eq:21}
\|A\|_u \eqdef \tfrac{1}{\sqrt{2}} \|A\|_\Frob \eqdef \sqrt{ \tfrac12 \mop{Tr}( A A^* ) }, \quad A\in \bC^{(n+1)\times (n+1)},
\end{equation}
where $A^*$ denote the conjugate transpose,
and we choose on~$U(n+1)$ the Riemannian metric induced from the embedding
of~$U(n+1)$ in~$\bC^{(n+1) \times (n+1)}$. This metric is invariant under left
and right multiplication in~$U(n+1)$.

Let~$X$ and~$Y$ be Riemannian manifolds and let~$f : X\to Y$ be
an infinitely differentiable surjective map.
For any~$x\in X$, we define the \emph{normal Jacobian} of~$f$ at~$x$ as
\begin{equation*}
  \mop{NJ}_x f \eqdef \sqrt{ \det\left( \ud_xf \cdot \ud_xf^* \right) }.
\end{equation*}
When~$\ud_x f$ is bijective, this is the absolute value of the usual Jacobian.
A fundamental result is the \emph{coarea formula} for Riemannian manifolds
\parencites[Theorem~3.1]{Federer_1959}[Appendix]{Howard_1993}: for any integrable
map~$\Theta : X \to \bR$,
\begin{equation}\label{eq:coarea}
  \smashint[\ud x]X\ \Theta(x) {\mop{NJ}_x f} = \smashint[\ud y]Y \smashintlong[\ud x]{f^{-1}(y)} {\Theta(x)}. 
\end{equation}
The special case of Riemannian submersions is important. We say that~$f$ is a
\emph{Riemannian submersion} if for any~$x\in X$, the derivative $\ud_x f$
induces an isometry from~$\left( \ker \ud_xf \right)^\perp$ to~$T_{f(x)} Y$. In
that case, we easily check that $f$ is Lipschitz-continuous with
constant~$1$ and that $\mop{NJ}_xf = 1$ for all~$x\in X$.
Note also that for any submanifold~$Z$ of~$Y$, if~$f$ is a Riemannian submersion
then so is~$f_{|f^{-1}(Z)}$.
The scaling in the definition of~$\|-\|_u$ is chosen to have the following result.

\begin{lemma}\label{lem:riem-submersion-proj}
  For any~$p\in \Proj$, the map~$\phi : u\in U(n+1)\mapsto up\in \Proj$
  is a Riemannian submersion.
  In particular, for any variety~$X\subseteq \Proj$,
  and any integrable map~$\Theta : \phi^{-1} X \to \bR$,
  \[ \smashint[\ud u]{\phi^{-1} X} \ \Theta(u) = \smashint[\ud x]{X}
    \smashint[\ud u]{up = x}\ \Theta(u), \]
  where $\int_{up=x} \ud u$ denotes the integration over the variety $\phi^{-1}(x)$.
\end{lemma}

\begin{proof}
  Thanks to the invariance of the Riemannian metric of~$U(n+1)$ under right
  multiplication, it is enough to check that the defining property of Riemannian
  submersion holds at~$\Id$, the identity matrix. With a suitable choice of coordinates, we may also
  assume that~$p = [1:0:\dotsb:0]$. The tangent space of~$\Proj$ at~$p$ is
  canonically identified with~$\left\{ p \right\}^\perp$, that is $\left\{ 0
  \right\}\times \bC^n$.

  The tangent space~$T_{\Id} U(n+1)$ of~$U(n+1)$ at~$\Id$ is the space
  of skew-Hermitian matrices, and for any~$u \in T_{\Id} U(n+1)$,
  $\ud_{\Id} \phi(\dot u) = \dot u p$.
  Therefore,
  \[ \left( \ker \ud_{\Id} \phi \right)^\perp = \left\{
      \begin{psmallmatrix}
        0 & v^* \\
        v & \mathbf{0}
      \end{psmallmatrix} \st v\in \bC^n
    \right\}, \]
  and since~$\left\|  \begin{psmallmatrix}
        0 & v^* \\
        v & \mathbf{0}
      \end{psmallmatrix} \right\|_u = \|v\|$, the map~$\dot u \in \left( \ker \ud_{\Id} \phi \right)^\perp \to \dot u p$ is
  clearly an isometry.

  The second claim follows from coarea formula~\eqref{eq:coarea}, noting that
  the restriction of~$\pi$ to~$\pi^{-1} X$ is again a Riemannian submersion.
\end{proof}


\subsection{Basic integral formulas}

For our problem, the manifold~$\cV$ has a natural distribution, the \emph{standard distribution},
denoted~$\rho_\text{std}$, defined as follows. Let~$\bfu \in \cU$ and~$x\in
\cap\bfu\cX$ be uniformly distributed, 
so that the random variable~$(\bfu, x)$ belongs
to~$\cV$ and let~$\rstd$ be its probability distribution.
The uniform distribution is defined on~$\cU$ by the Riemannian
metric; and on~$\cap \bfu \cX$ by the Riemannian metric on the regular locus.
This section aims at describing the conditional probability distribution
of~$(\bfu, x)$, given~$x$ and the linearization~$L(\bfu, x)\in\cL$.

For any~$x\in\Proj$, $\bfy \in \cX$ and~$\bfh \in \cL$, we define
\begin{align*}
  \cL_x &\eqdef \left\{ \bfh \in \cL \st x \in \cap \bfh \right\}, \\
  \cU_x &\eqdef \left\{ \bfu \in \cU \st x \in \cap \bfu \cX \right\} \\
  \cU_{x,\bfy} &\eqdef \left\{ \bfu\in\cU_x \st \bfu \bfy = (x, \dotsc, x)
  \right\}, \\ 
  \cU_{x,\bfy,\bfh} &\eqdef \left\{ \bfu \in \cU_{x,\bfy}
                      \st L(\bfu,x) = \bfh \right\}.
\end{align*}

\begin{lemma}
  \label{lemma:basic-integral-formula}
  For any two submanifolds~$X$ and~$Y$ of\/~$\Proj$, with~$\codim X + \codim Y \leq n$, and for any integrable function~$\Theta :  U(n+1)\times\Proj \to \bR$,
  \[ \smashintlong[\ud u]{U(n+1)} \smashint[\ud z]{X\cap u Y} \ \Theta(u, z)  =
    \smashint[\ud x]X \smashint[\ud y]Y \smashint[\ud u]{uy=x} \ \Theta(u, x)
    \det\nolimits^\perp(\mathbb{T}_xX,\mathbb{T}_x uY), \]
  (where $\int_{uy=x} \ud u$ denotes the integration over the variety of
  all~$u\in U(n+1)$ such that~$uy=x$, as in Lemma~\ref{lem:riem-submersion-proj}).
\end{lemma}

\begin{proof}
  It is  a corollary of the “basic integral formula” of \textcite[\S2.7]{Howard_1993}.
  Let~$p \in \Proj$ be some point and let
  $\phi$ be the Riemannian submersion $u \in U(n+1) \mapsto u p \in \Proj$
  (Lemma~\ref{lem:riem-submersion-proj}).
  Let~$M = \phi^{-1}X$ and~$N=\phi^{-1}Y$.
  Howard's basic integral formula asserts that
  \begin{equation}\label{eq:8}
   \smashintlong[\ud u]{U(n+1)} \smashint[\ud v]{M \cap u N}\ \Theta(u, \phi v)
    = \smashint[\ud v]M\smashint[\ud w]N\ \Theta(v w^{-1}, \phi v)
    \det^\perp(v^{-1} T_v M, w^{-1} T_w N),
  \end{equation}
  where $v^{-1} T_v M$ and $w^{-1} T_w N$ are subspaces of~$T_{\Id} U(n+1)$.
  (The equality of our~$\det^\perp$ with Howard's $\sigma$ is given by Lemma~\ref{lem:multiproj-gram}.)
  On the one hand, by Lemma~\ref{lem:riem-submersion-proj}, we obtain the
  following expression for the left-hand side of~\eqref{eq:8}:
  \begin{align*}
    \smashintlong[\ud u]{U(n+1)} \smashint[\ud v]{M \cap u N}\ \Theta( u,\phi v) &= \smashintlong[\ud u]{U(n+1)}\smashintlong[\ud z]{X \cap uY}\smashint[\ud v]{vp=z}\ \Theta(u,z)  \\
    &= \smashintlong[\ud u]{U(n+1)}\smashint[\ud z]{X \cap uY}\  \Theta(u,z) \mop{vol} \left\{ v\in U(n+1) \st vp=z \right\}  \\
    &=\mop{vol}(U(1)\times U(n)) \smashintlong[\ud u]{U(n+1)}\smashint[\ud z]{X \cap uY}\ \Theta(u,z),
  \end{align*}
  where, for the last equality, we remark that~$\left\{ v\in U(n+1) \st vp=z \right\}$ is isometric, by some
  left multiplication, to $U(1)\times U(n)$,
  the stabilizer of a point in~$\Proj$. This is the left-hand side of the claimed
  equality. On the other hand, regarding the right-hand side of~\eqref{eq:8},
  we check easily that
  \begin{align*}
    \det^\perp(v^{-1} {T}_v M, w^{-1} {T}_w N) &= \det^\perp ( v^{-1} \mathbb{T}_{vp} X, w^{-1} \mathbb{T}_{wp} Y )\\
                                               &= \det^\perp\left( \mathbb{T}_{vp} X, \mathbb{T}_{vp}(v w^{-1} Y) \right)
  \end{align*}
  and therefore (using Lemma~\ref{lem:riem-submersion-proj} again) 
  \begin{align*}
    \smashint[\ud v]M\smashint[\ud w]N\ &\Theta(v w^{-1}, \phi v) \det^\perp(v^{-1} T_v M, w^{-1} {T}_w N)\\
    &= \smashint[\ud x]X \smashintlong[\ud v]{vp=x}\smashint[\ud y]Y\smashint[\ud w]{wp=y}\ \Theta(v w^{-1}, x) \det^\perp(\mathbb{T}_x X, \mathbb{T}_x vw^{-1} Y)\\
    &= \mop{vol}(U(1)\times U(n)) \smashint[\ud x]X \smashint[\ud y]Y \smashint[\ud u]{uy=x}\ \Theta(u, x) \det^\perp(\mathbb{T}_x X,\mathbb{T}_x uY),
  \end{align*}
  where the last equality is given by the change of variables $v = u w$. This
  gives the right-hand side of the claim.
\end{proof}

In our setting where we consider $r$ varieties $X_1, \dotsc, X_r$, we
can give the following ``basic integral formula''. It has been proved very
similarly in the real case by \textcite[\S7.5]{BurgisserLerario_2018}. 

\begin{proposition}
  \label{prop:multi-basic-integral-formula}
  For any measurable function  $\Theta : \cV \to [0,\infty)$
  \[ \smashint[\ud \bfu]{\cU}\smashint[\ud x]{\cap\bfu\cX}\  \Theta(\bfu, x)
  = \smashint[\ud x]{\Proj} \smashint[\ud \bfu]{\cU_x}\  \Theta(\bfu, x) \det^\perp(L(\bfu, x)). \]
\end{proposition}

\begin{proof}
  We proceed by induction on~$r$.
  When $r=1$,
  \begin{align*}
    \smashint[\ud \bfu]{\cU}\smashint[\ud x]{\cap\bfu\cX}\  \Theta(\bfu, x)
    &= \smashintlong[\ud u]{U(n+1)}\smashint[\ud x]{u X_1}\  \Theta(u, x)\\
    &= \smashint[\ud x]{\Proj}\smashint[\ud y]{X_1}\smashint[\ud u]{uy=x}\ \Theta(u, x) \det^\perp(\mathbb{T}_x \Proj, \mathbb{T}_x uX_1), 
  \end{align*}
  by Lemma~\ref{lemma:basic-integral-formula} with $X = \Proj$ and $Y = X_1$.
  Note that~$\det^\perp(\mathbb{T}_x \Proj, \mathbb{T}_x uX_1) =
  \det^\perp(\mathbb{T}_x uX_1) = 1$, and it follows that
  \begin{align*}
    \smashint[\ud \bfu]{\cU}\smashint[\ud x]{\cap\bfu\cX}\  \Theta(\bfu, x)
     &=  \smashint[\ud x]{\Proj}\smashint[\ud y]{X_1}\smashint[\ud u]{uy=x}\ \Theta(u, x)\\
     &=  \smashint[\ud x]{\Proj}\smashint[\ud y]{X_1}\smashint[\ud v]{y=vx}\ \Theta(v^{-1}, x), && \text{by the change of variable $v=u^{-1}$,}\\
     &= \smashint[\ud x]{\Proj} \smashintlong[\ud v]{x\in v^{-1}X_1}\ \Theta(v^{-1}, x), &&\text{by Lemma~\ref{lem:riem-submersion-proj},}\\
     &= \smashint[\ud x]{\Proj} \smashintlong[\ud u]{\cU_x}\ \Theta(u, x), &&\text{by the change of variable $u = v^{-1}$.}
  \end{align*}
  This concludes the base case
  since~$\det^\perp(L(\bfu, x))$ is identically~$1$ when~$r=1$. 
   
  Assume that the property holds for~$r-1$ subvarieties~$X_1,\dotsc,X_{r-1}$,
  for some~$r\geq 1$,
  and let~$\cU'$, $\cX'$, \mbox{etc.} denote the analogues of~$\cU$, $\cX$,
  \mbox{etc.} for the varieties $X_1,\dotsc,X_{r-1}$.
  From the decomposition~$\cU = \cU'\times U(n+1)$, we obtain by Lemma~\ref{lemma:basic-integral-formula}
  \begin{align*}
    &\smashint[\ud \bfu]{\cU}\smashint[\ud x]{\cap\bfu\cX}\  \Theta(\bfu, x)
    = \smashint[\ud\bfu']{\cU'}  \smashintlong[\ud u_r]{U(n+1)} \smashint[\ud x]{(\cap\bfu'\cX') \cap u_r X_r}\  \Theta(\bfu', u_r, x) \\
    &= \smashint[\ud\bfu']{\cU'}\smashintlong[\ud x]{\cap\bfu'\cX'}\smashint[\ud y]{X_r} \smashint[\ud u_r]{u_r y = x} \ \Theta(\bfu', u_r, x) \det^\perp(\mathbb{T}_x (\cap \bfu'\cX'), \mathbb{T}_x u_r X_r)  \\
      &= \smashint[\ud x]{\Proj} \smashintlong[\ud \bfu']{\cU'_x} \smashint[\ud y]{X_r} \smashint[\ud u_r]{u_r y = x} \ \Theta(\bfu', u_r, x) \det^\perp(\mathbb{T}_x (\cap\bfu'\cX'), \mathbb{T}_x u_r X_r) \det^\perp(L(\bfu', x)),
  \end{align*}
  using the induction hypothesis for the last equation.
  Lemma~\ref{lem:multiangle} shows that
  \[ \det^\perp(\mathbb{T}_x (\cap\bfu'\cX'), \mathbb{T}_x u_r X_r) \det^\perp(L(\bfu', x)) = \det^\perp (L(\bfu,x)). \]
  Moreover, by Lemma~\ref{lem:riem-submersion-proj} (applied as in the base case), 
  \[ \smashint[\ud y]{X_r} \smashint[\ud u_r]{u_r y = x} \ h(u_r) =
    \smashint[\ud u_r]{x \in u_r X_r} \ h(u_r), \]
  for any integrable function~$h : U(n+1)\to \bR$.
  This implies that
  \begin{align*}
    \smashint[\ud \bfu]{\cU}\smashint[\ud x]{\cap\bfu\cX}&\  \Theta(\bfu, x)    
      = \smashint[\ud x]{\Proj} \smashintlong[\ud \bfu']{\cU'_x} \smashint[\ud u_r]{x \in u_r X_r} \ \Theta(\bfu', u_r, x) \det_x^\perp(L(\bfu,x)).
  \end{align*}
  To conclude, we remark that~$\cU_x = \cU'_x\times\left\{ u_r\in U(n+1) \st x\in u_r X_r \right\}$.
\end{proof}

If we apply the statement above to the case where the varieties~$X_i$ are
projective subspaces, that is~$X_i \in \mathbb{G}(\dim X_i)$, we obtain the following corollary.

\begin{corollary}\label{coro:multi-lin}
  For any  measurable function  $\Theta : \left\{(\bfh, x)\in\cL \st x \in \cap \bfh\right\} \to [0,\infty)$
  \[ \smashint[\ud \bfh]{\cL}\smashint[\ud x]{\cap\bfh}\  \Theta(\bfh,x)
  = \smashint[\ud x]{\Proj} \smashint[\ud \bfh]{\cL_x}\  \Theta(\bfh, x) \det^\perp(\bfh). \]
\end{corollary}

\begin{proof}
  Let us assume that each~$X_i$ is a projective subspace of~$\Proj$. The
  map
  \[ \bfu \in \cU \mapsto \bfu \cX = (u_1 X_1,\dotsc,u_r X_r) \in \cL, \]
  is a Riemannian submersion,
  thus
  \begin{align*}
    \smashint[\ud \bfu]{\cU}\smashint[\ud x]{\cap\bfu\cX}\  \Theta(\bfu\cX, x)
    &= \smashint[\ud \bfh]{\cL}\smashintlong[\ud \bfu]{\bfu \cX = \bfh}\smashint[\ud x]{\cap\bfh}\  \Theta(\bfh, x)\\
    &= \mop{vol}(\mop{Stab}_\cU \cX)\smashint[\ud \bfh]{\cL} \smashint[\ud x]{\cap \bfh}\  \Theta(\bfh, x),
  \end{align*}
  where $\mop{Stab}_\cU \cX = \left\{ \bfu \in \cU \st \bfu\cX = \cX \right\}$,
  because all the subsets~$\left\{ \bfu \in \cU \st \bfu\cX = \bfh \right\}$ are
  isometric, by some left multiplication, to~$\mop{Stab}_\cU \cX$.
  Similarly,
  \begin{equation*}
    \smashint[\ud x]{\Proj} \smashint[\ud \bfu]{\cU_x}\  \Theta(\bfu \cX, x) \det^\perp(L(\bfu, x))
    = \mop{vol}(\mop{Stab}_\cU \cX) \smashint[\ud x]{\Proj} \smashint[\ud \bfh]{\cL_x}\  \Theta(\bfh, x) \det^\perp(\bfh ).
  \end{equation*}
  This reduces the claim to Proposition~\ref{prop:multi-basic-integral-formula}.
\end{proof}

We now have all we need to prove the main result of this section.

\begin{theorem}
  For any measurable function~$\Theta : \cV \to [0,\infty)$,
  \[ \smashint[\ud \bfu]{\cU}\smashint[\ud x]{\cap\bfu\cX}\  \Theta(\bfu, x)
  = \smashint[\ud {\bfh}]{\cL}\smashint[\ud\bfy]{\cX} \smashint[\ud x]{\cap \bfh} \smashint[\ud \bfu]{\cU_{x,\bfy,\bfh}}\ \Theta(\bfu, x). \]
  In other words, if~$(\bfu,x)\in\cV$ is a $\rho_\text{std}$-distributed random variable distributed, then:
  \begin{enumerate}[(i)]
  \item the random variable $L(\bfu,x)$ is uniformly distributed in~$\cL$;
  \item \label{item:8} the random variable $\bfy \eqdef (u_1^{-1} x, \dotsc, u_r^{-1} x)$ is uniformly
    distributed in~$\cX$ and independent from~$L(\bfu,x)$;
  \item conditionally on~$L(\bfu, x)$, the random variable $x$ is uniformly distributed in~$\cap L(\bfu,x)$;
  \item conditionally on~$L(\bfu,x)$, $x$ and~$\bfy$, the random variable $\bfu$ is uniformly distributed in~$\cU_{x,\bfy,L(\bfu,x)}$.
  \end{enumerate}

  \label{thm:linearization}
\end{theorem}

\begin{proof}
  By Corollary~\ref{coro:multi-lin},
  \begin{align*}
    \smashint[\ud {\bfh}]{\cL}\smashint[\ud x]{\cap \bfh} \smashint[\ud\bfy]{\cX} \smashint[\ud v]{\cU_{x,\bfy,\bfh}}\ \Theta(\bfu, x)
  &= \smashint[\ud x]{\Proj} \smashint[\ud\bfh]{\cL_x} \det^\perp(\bfh) \smashint[\ud\bfy]{\cX} \smashint[\ud v]{\cU_{x,\bfy,\bfh}}\ \Theta(\bfu, x) \\
  &= \smashint[\ud x]{\Proj} \smashint[\ud\bfy]{\cX} \smashint[\ud\bfh]{\cL_x} \smashint[\ud v]{\cU_{x,\bfy,\bfh}}\ \Theta(\bfu, x)\det^\perp(L(\bfu,x)).
  \end{align*}
  Moreover, the map~$\bfu \in \cU_{x,\bfy} \mapsto L(\bfu,x) \in \cL_x$ is a Riemannian submersion, thus
  \[ \smashint[\ud\bfh]{\cL_x} \smashint[\ud v]{\cU_{x,\bfy,\bfh}}\ \Theta(\bfu, x)\det^\perp(L(\bfu,x)) = \smashint[\ud v]{\cU_{x,\bfy}}\ \Theta(\bfu, x)\det^\perp(L(\bfu,x)). \]
  The map~$\bfu \in \cU_{x} \mapsto (u_1^{-1} x, \dotsc, u_m^{-1} x)\in \cX$ is also a Riemannian submersion, thus
  \[ \smashint[\ud\bfy]{\cX}\smashint[\ud v]{\cU_{x,\bfy}}\ \Theta(\bfu, x)\det^\perp(L(\bfu,x)) = \smashint[\ud v]{\cU_x} \ \Theta(\bfu, x)\det^\perp(L(\bfu,x)). \]
  To conclude, we apply Proposition~\ref{prop:multi-basic-integral-formula}.
\end{proof}

\begin{corollary}\label{coro:uniform-sampling}
  Let~$X$ be subvariety of\/~$\Proj$
  and~$L \in \mathbb{G}(\codim X)$ be a uniformly distributed random projective subspace.
  Let~$\zeta \in X \cap L$ be a uniformly distributed random variable.
  Then~$\zeta$ is uniformly distributed in~$X$.
\end{corollary}

\begin{proof}
  Let~$L_0 \in \mathbb{G}(\codim X)$ and let~$v\in U(n+1)$ a be uniformly distributed random variable.
  The random subspace~$vL_0$ is uniformly distributed in~$\mathbb{G}(\codim X)$:
  indeed, the probability distribution of~$vL_0$ is invariant under the action
  of~$U(n+1)$ and, by transitivity of the action of~$U(n+1)$
  on~$\mathbb{G}(\codim X)$, there is a unique invariant probability
  distribution on~$\mathbb{G}(\codim X)$. So we may assume that~$L = v L_0$.

  Let~$u \in U(n+1)$ be an independent uniformly distributed random variable.
  The random variables~$u$ and~$v'\eqdef uv$ are independent and uniformly distributed,
  because the diffeomorphism
  \[ (u,v) \in U(n+1) \times U(n+1) \mapsto (u, u v) \in U(n+1) \times
    U(n+1) \]
  has constant Jacobian, and so preserves the uniform distribution.
  Moreover, $u\zeta$ is uniformly distributed in~$u X \cap v' L_0$.
  Therefore, the pair~$\left( (u,v'), u\zeta \right)$ is~$\rstd$-distributed in
  the solution variety associated to~$X$ and~$L_0$.
  By Theorem~\ref{thm:linearization}\ref{item:8}, $u^{-1} (u \zeta)$ is
  uniformly distributed in~$X$, which is the claim.
\end{proof}

\subsection{Sampling the solution variety}
\label{sec:sampl-solut-vari}

\begin{algo}[tp]
  \centering
  \begin{algorithmic}
    \Function{Sample}{$X_1,\dotsc,X_r$}
    \State Sample $h_1 \in \mathbb{G}(\dim X_i), \dotsc, h_r \in \mathbb{G}(\dim
    X_r)$, uniformly and independently.
    \State Sample $\zeta \in h_1 \cap \dotsb \cap h_r \subset \Proj$ uniformly.
    \State Sample $y_1 \in X_1,\dotsc,y_r\in X_r$ uniformly and independently.
    \State Sample $u_1,\dotsc,u_r \in U(n+1)$, such that~$u_i y_i= \zeta$ 
    and~$u_i(\mathbb{T}_{y_i} X_i) = h_i$,
    uniformly and independently.
    \State \textbf{return} $(u_1,\dotsc,u_r) \in \cU$ and~$\zeta \in \Proj$.
    \EndFunction
  \end{algorithmic}
  \caption[]{Sampling of a unitary solution variety
    \begin{description}
      \item[Input.] Varieties $X_1,\dotsc,X_r \subset \Proj$ with~$\sum_i \codim
        X_i \leq n$. 
      \item[Output.] $(\bfu, \zeta) \in \cV$, where~$\cV = \left\{
          (u_1,\dotsc,u_r, x)\in \cU \times \Proj \st x \in \cap_i u_i X_i \right \}$.
      \item[Postcondition.] $(\bfu, \zeta) \sim \rho_\text{std}$ (Theorem~\ref{thm:linearization})
    \end{description}
  }
  \label{algo:sample}
\end{algo}

Based on Theorem~\ref{thm:linearization}, Algorithm~\ref{algo:sample} samples a
$\rho_\text{std}$-distributed random~$(\bfu, \zeta)\in\cV$.
We explain briefly how to perform the four steps of the algorithm.

For each~$1\leq i\leq r$, we sample independently linear
forms~$\lambda_{i,1},\dotsc,\lambda_{i,\mop{codim} X_i} \in (\bC^{n+1})^*$ with a
standard normal distribution.
We define $h_i$ as the
zero locus of~$\lambda_{i,1},\dotsc,\lambda_{i,\mop{codim} X_i}$.
Next, we compute a unitary basis of the linear subspace
$h_1\cap \dotsb \cap h_n$ and use it to
sample~$\zeta \in \bP(h_1\cap \dotsb \cap h_n)$ with a uniform distribution.

To sample uniformly a point~$y_i \in X_i$, we consider a random uniformly
distributed subspace~$L_i\in \mathbb{G}(\mop{codim} X_i)$. Almost surely, the
intersection~$X_i\cap L_i$ is finite and we sample uniformly a point~$y_i$ in
it. By Corollary~\ref{coro:uniform-sampling}, 
$y_i$ is uniformly distributed in~$X_i$.
Since~$L_i$ is a projective subpace, the computation of~$X_i\cap L_i$ requires a polynomial system solving in~$\mop{codim} X_i + 1$ homogeneous variables.
In the typical case where~$X_i$ is a hypersurface defined by a polynomial~$f_i$
and~$L_i$ is a projective line, this amounts to compute the zeros~$[x:y]\in
\bP^1$
of the homogeneous equation~$f_i(xp+yq) = 0$, for some basis~$\left\{ p,q
\right\}$ of~$L_i$.

Once we get the $y_i$, we compute, for each~$1\leq i\leq r$,
 some~$v_i \in U(n+1)$ which maps~$y_i$ to~$\zeta$ and~$\mathbb{T}_{y_i} X_i$
to~$h_i$ and we sample uniformly a~$w_i$ in the subgroup of all~$w\in U(n+1)$
such that~$w \zeta = \zeta$ and~$w L_i = L_i$. This subgroup is isometrically
isomorphic to~$U(1) \times U(\dim X_i) \times U(\mop{codim} X_i)$. We can sample
uniformly in a unitary group~$U(k)$ by considering the~$Q$ factor of the QR
decomposition of a random $k\times k$~Gaussian matrix. And then, we define~$u_i \eqdef w_i
v_i$.
When~$X_1,\dotsc,X_r$ are all hypersurfaces, we have proved the
following proposition.
\begin{proposition}\label{prop:complexity-sampling}
  If~$X_1,\dotsc,X_r$ are all hypersurfaces, Algorithm~\ref{algo:sample} samples a $\rho_\text{std}$-distributed point in the
  solution variety~$\cV$ with
  \begin{itemize}
  \item $r$ times root-finding of bivariate homogeneous polynomials of respective degrees $\deg
    X_1$,\ldots, $\deg X_r$;
  \item $O(n^3)$ samplings of the standard normal distribution on~$\bR$; and
  \item $O(n^4)$ arithmetic operations.
  \end{itemize}
\end{proposition}

\subsection{The split gamma number}
\label{sec:splitgamma}

In the classical theory, the condition number~$\mu$ plays two roles: first, by
definition, it bounds the variation of a zero after a pertubation of the system;
second, it is an upper bound for Smale's gamma number with some regularity properties (the
Lipschitz properties). Each role is reflected by a
factor~$\mu$ in the $\mu^2$ estimate.

In the rigid setting, the two roles are played by different numbers: the
variation of a zero is bounded by the incidence condition number $\kappa$
and the upper bound for $\gamma$ that we use is the \emph{split gamma number}
$\hat\gamma$. This will lead to a $\kappa \hat\gamma$ estimate for the
complexity of numerical continuation in the rigid setting. In this section, we
introduce the split gamma number and we start with some reminders about the gamma
number.


Let~$F = (f_1,\dotsc,f_s) \in \cH[s]$ be a homogeneous polynomial system 
that we regard as a polynomial
map~$\bC^{n+1}\to \bC^s$. Let~$d_i \eqdef \deg f_i$ and~$D\eqdef \max_i d_i$.

When~$s=n$, the polynomial system~$F$ has generically finitely many zeros and
our primary goal is to compute them numerically and approximately. A fundamental
tool is Newton's operator. We use here the projective version introduced by
\textcite{Shub_1989}. For~$z \in \Proj$, projective class of some~$\bar
z\in\bC^{n+1}$, we define
\begin{equation}\label{eq:22}
  \cN_F(z) \eqdef \left[ \bar z - \ud_{\bar z} F|_{z^\perp}^{-1} \left( F(\bar z) \right) \right] \in \Proj, 
\end{equation}
where~$z^\perp$ is the orthogonal complement of~$z$ in~$\bC^{n+1}$. The
definition does not depend on the choice of~$\bar z$. Given a zero~$\zeta \in
\Proj$ of~$F$, we say that~$z\in\Proj$ \emph{approximates $\zeta$ as a zero of~$F$},
or that \emph{$z$ is an approximate zero of~$F$
  with associated zero~$\zeta$}, if for any~$k\geq 0$,
\[ d_\bP\left(\cN_F^k(z), \zeta \right) \leq 2^{1-2^k} d_\bP(z,\zeta), \]
where~$d_\bP$ is the geodesic distance in~$\Proj$, see \S \ref{sec:notat-basic-defin}.

The main result of the gamma theory is a sufficient condition for a point to
approximate a zero. For a polynomial system~$F$, in the general case~$r \leq n$,
we will use the following definition \parencites{ShubSmale_1996}[\S4]{Dedieu_2006} for the gamma number of~$F$ at $z \in
\Proj$:
\begin{equation}\label{eq:24}
 \gamma(F,z) \eqdef
  \begin{cases}
  \sup_{k\geq 2} \tn{\tfrac{1}{k!}\ud_zF^\dagger \cdot \ud_z^kF
    }^{\frac{1}{k-1}} & \text{if $\ud_zF$ is
    surjective,} \\
  \infty & \text{otherwise.}
  \end{cases}
\end{equation}
(The definition does not depend on the choice of a unit representative~$\bar z$
of~$z$.)
When~$s=n$, the pseudo-inverse~$\ud_z F^\dagger$ is
often replaced by $\ud_z F|_{z^\perp}^{-1}$, as in Newton's iteration
\parencite[\emph{e.g.}][]{ShubSmale_1994,BurgisserCucker_2013}. If~$z$ is a zero
of~$F$, both definitions coincide, so the gamma theorem
(Theorem~\ref{thm:gamma}) is equally true for both variants.

When~$F = (f)$ is a single equation, that is~$s = 1$, it is useful to remark
that~$\ud_z f$ is a linear form and so~$\ud_z f^\dagger$ is simply $\| \ud_z f
\|^{-1} \pi$, where~$\pi$ is an isometric embedding of~$\bC$ in~$\bC^{n+1}$.
This gives~$\gamma(f,z)$ the following form:
\begin{equation}\label{eq:15}
 \gamma(f,z) = \sup_{k\geq 2} \left( \tfrac{1}{k!} \| \ud_zf \|^{-1} \VERT \ud_z^kf
   \VERT \right)^{\frac{1}{k-1}}.
\end{equation}

The following lower bound will be occasionally useful.
\begin{lemma}
  For any~$z\in\Proj$, $\gamma(F,z) \geq \frac{D-1}{2}$.
  \label{lem:gamma-lower-bound}
\end{lemma}

\begin{proof}
  We may assume that~$\ud_zF$ is surjective, otherwise the bound is trivial.
  Let~$d_i \eqdef \deg f_i$. Using the homogeneity, for any $u\in \bC^{n+1}$,
  \begin{equation}\label{eq:20}
   \ud_z^2F(z, u) =
    \left(\begin{smallmatrix}
        d_1 - 1 & & \\
        & \ddots & \\
        & & d_m-1
      \end{smallmatrix}\right) \ud_z F(u).
  \end{equation}

  We fix some $1\leq i\leq r$ and consider $u \eqdef \ud_zF^\dagger (e_i)$, where
  $e_i \eqdef (0,\dotsc,1,\dotsc,0)\in\bC^r$ with a single one at the
  $i$th position.
  Becase~$\ud_zF$ is surjective, $\ud_zF(\ud_zF^\dagger (e_i)) = e_i$,
  and by~\eqref{eq:20}, we have $\ud_z^2 F(z, u) = (d_i-1) e_i$ and then
  $\ud_zF^\dagger( \ud_z^2 F(z,u)) = (d_i-1)u$.
  This implies that~$\tn{\ud_zF^\dagger \ud_z^2 F} \geq d_i-1$, for any~$1\leq i
  \leq m$, and the claim follows.
\end{proof}

We now state the main result of the gamma theory in the projective setting, primarily due to \textcite{ShubSmale_1993b}.
\begin{theorem}[Shub, Smale]
  \label{thm:gamma}
  Let~$F\in \cH[n]$ be a homogeneous polynomial system.
  For any zero~$\zeta\in\Proj$ of~$F$ and any~$z\in\Proj$,
  if $d_\bP(z,\zeta)\gamma(F,\zeta) \leq \frac{1}{6}$
  then~$z$ approximates $\zeta$ as a zero of~$F$.
\end{theorem}


\begin{proof}
  This is \parencite[Theorem~16.38]{BurgisserCucker_2013} with~$r=0.981$ (and we use
  that~$\gamma(F,\zeta)\geq \frac12$ when~$D \geq 2$, Lemma~\ref{lem:gamma-lower-bound}).
\end{proof}

Let~$F_1 \in \cH[s_1],\dotsc,F_r \in \cH[s_r]$ be homogeneous polynomial
systems, with~$s_1+\dotsb+s_r \leq n$.
Based on the incidence condition number of linear subspaces (\S \ref{sec:orth-determ-incid}),
we define the \emph{incidence condition number} of~$F_1,\dotsc,F_r$ at a
point~$x \in \Proj$ by
\begin{equation}\label{eq:38}
  \kappa(F_{1},\dotsc,F_r ; x) \eqdef \kappa\left( \ker(\ud_x F_1)^\perp, \dotsc,
    \ker( \ud_x F_r)^\perp \right).
\end{equation}
Since the orthogonal projector on~$\ker(\ud_x F_i)^\perp$ is~$\ud_x F_i^\dagger
\circ \ud_x F_i$, we can write
\begin{equation}
  \kappa(F_{1},\dotsc,F_r ; x) = \tn{
    \begin{pmatrix}
      \ud_x F_1^\dagger
      \circ \ud_x F_1 \\
      \hline
      \vdots \\
      \hline
      \ud_x F_r^\dagger
\circ \ud_x F_r
    \end{pmatrix}^\dagger
}.
\end{equation}
The \emph{split gamma number} of~$F_1,\dotsc,F_r$ at a
point~$x \in \Proj$ is defined by
\begin{equation}\label{eq:37}
  \hat\gamma(F_{1},\dotsc,F_r ; x) \eqdef \kappa( F_{1},\dotsc,F_r ; x) \left(
    \gamma(F_1, x)^2 + \dotsb + \gamma(F_r, x)^2 \right)^{\frac12}.
\end{equation}

The split gamma number separates the contribution of the $\gamma$ number of each
block of equations from the more geometric information contained in~$\kappa$:
for~$x\in \cap_i V(F_i)$,
the~$\kappa$ factor only depends on the relative position of the tangent spaces
of the varieties~$V(F_1),\dotsc,V(F_r)$ at~$x$; while the~$\gamma$ factor
quantifies how much each~$V(F_i)$ deviates from its tangent space at~$x$.
Note that when~$r=1$, then~$\hat\gamma(F,x) = \gamma(F,x)$.

\begin{theorem}\label{thm:gamma-leq-sgamma}
  Let $G \eqdef (F_1,
\dotsc, F_r) \in \cH[s_1+\dotsb+s_r]$ denote the concatenation of the systems.
  For any~$x\in \Proj$,
  \[ \gamma(G,x)
    \leq \hat\gamma(F_1, \dotsc, F_r ; x)
    \leq r\, \kappa(F_{1},\dotsc,F_r ; x) \gamma(G,x). \]
\end{theorem}

\begin{proof} It is easy to see that
  $\hat\gamma(F_1,\dotsc,F_r;x)$ is finite if and only if~$\gamma(G, x)$ is. Thus,
  we may assume that~$\ud G$ and the~$\ud F_i$ are surjective (we drop the
  index~$x$ in~$\ud_x$).
  We begin with the first inequality.
  Let~$K_i \eqdef \ker \ud_x F_i$ and $P \eqdef
  \operatorname{proj}(K_1^\perp,\dotsc,K_r^\perp)$,
  so that~$\kappa(F_1,\dotsc,F_r ; x) = \tn{P^\dagger}$.
  We first prove that for any~$k\geq 2$ and any~$\mathbf{y} =
  (y_1,\dotsc,y_k) \in \left( \bC^{n+1} \right)^k$,
  \begin{equation}\label{eq:6}
    \ud G^\dagger \cdot \ud^k G(\mathbf{y}) = P^\dagger \left(\ud F_1^\dagger \cdot \ud^k F_1(\mathbf y),\dotsc,\ud F_r^\dagger \cdot \ud^k F_r(\mathbf y) \right).
  \end{equation}
  It is clear that
  \begin{equation*}
    \ud^k G(\mathbf{y}) = \left(\ud^k F_1(\mathbf y),\dotsc,\ud^k F_r(\mathbf y) \right) \in \bC^{s_1} \times \dotsb \times \bC^{s_r}.
  \end{equation*}
  Let~$v_i \eqdef \ud^k F_i(\mathbf y)$ and $\mathbf v \eqdef (v_1,\dotsc,v_r)$.
  Because $\ud G$ is surjective, we have~$\ud G \cdot \ud G^\dagger (\mathbf v) =
  \mathbf v$ and, equivalently, $\ud F_i \cdot \ud G^\dagger(\mathbf v) = v_i$.
  Therefore $\ud F_i^\dagger \cdot \ud F_i \cdot \ud G^\dagger \mathbf v = \ud F_i^\dagger
  v_i$. But~$\ud F_i^\dagger \cdot  \ud F_i$ is simply the orthogonal projection
  on~$K_i^\perp$. This gives $P \cdot \ud G^\dagger \cdot \ud^k G(\mathbf{y}) =
  \left(\ud F_1^\dagger v_1,\dotsc,\ud F_r^\dagger v_r \right)$, and since the
  image of~$\ud G^\dagger$ is orthogonal to the kernel of~$P$ we have
  $P^\dagger \cdot P \cdot \ud G^\dagger = \ud G^\dagger$. This proves
  \eqref{eq:6}.

  As a consequence, for any~$k \geq 2$,
  \begin{align*}
    \tn{ \tfrac{1}{k!} \ud G^\dagger \cdot \ud^k G }^{\frac{1}{k-1}}
    &\leq \tn{P^{\dagger}}^{\frac{1}{k-1}} \bigg( \sum_i \tn{ \tfrac{1}{k!} \ud F_i^\dagger \cdot \ud^k F_i }^2 \bigg)^{\frac{1}{2k-2}} \\
    &\leq \tn{P^{\dagger}}^{\frac{1}{k-1}} \bigg( \sum_i \gamma(F_i,x)^{2k - 2}  \bigg)^\frac{1}{2k-2} \\ 
    &\leq \tn{P^{\dagger}}  \bigg( \sum_i \gamma(F_i,x)^{2}  \bigg)^\frac{1}{2} 
  \end{align*}
  using~$\tn{P^{\dagger}} \geq 1$ (Lemma~\ref{lem:kappa-geq-1}) and the
  monotonicity of $p$-norms with respect to~$p$.
  By definition, $\kappa(F_{1},\dotsc,F_{r};x) = \tn{P^{\dagger}}$, so we obtain the
  first inequality.
  
  Concerning the second inequality, Equation~\eqref{eq:6} implies that
  \begin{align*}
    \tn{ \tfrac{1}{k!} \ud F_i^\dagger \cdot \ud^k F_i } &\leq \tn{P} \tn{ \tfrac{1}{k!} \ud G^\dagger \cdot \ud^k G } \\
                                                             &\leq \tn{P} \gamma(G,x)^{k-1}.
  \end{align*}
  We note that~$\tn{P} \leq r^{\frac12}$ (as the direct sum of~$r$ orthogonal
  projectors with orthogonal images) and therefore $\gamma(F_i,x) \leq
  r^{\frac12} \gamma(G,x)$. Hence
  \[ \left(\sum_{i=1}^r \gamma(F_i,x)^2\right)^{\frac12} \leq  \left( r \cdot 
      r \gamma(G,x)^2 \right)^{\frac12} \leq  r \gamma(G,x), \]
  and the second inequality follows.
\end{proof}

\section{Numerical continuation}
\label{sec:numer-cont-algor}

In this part, we consider a more specific setting than in the previous one.
We are given a polynomial system~$F = (f_1,\dotsc,f_n) \in \cH[n]$, with nonzero
square-free polynomials~$f_1,\dotsc,f_n$,
and a $\bfu\in \cU = U(n+1)^n$, and we look for a zero of the polynomial system~$\bfu \cdot F \eqdef (f_1 \circ
u_1^{-1},\dotsc,f_n\circ u_n^{-1})$. We will perform the average analyses in the case
where~$\bfu$ is uniformly distributed and~$F$ is fixed.
To this end, we consider the rigid solution variety~$\cV$ relative to the projective
hypersurfaces~$V(f_1),\dotsc,V(f_n)$.
In concrete terms, we have~$r=n$ and
\[ \cV = \left\{ (\bfu, x)\in \cU\times\Proj \st f_1(u_1^{-1} x) = \dotsb = f_n(u_n^{-1}
  x) = 0 \right\}. \]
As the Grassmannian of hyperplanes $\mathbb{G}(n-1)$ is isomorphic to~$\Proj$,
the manifold~$\cL$ is just~$(\Proj)^n$
and~$L(\bfu, x)$ is the $n$-uple
\[ L(\bfu, x) = \left( [\ud_x(f_1\circ u_1^{-1})], \dotsc, [\ud_x(f_n\circ
    u_n^{-1})]\right). \]
In particular, we can define $L(\bfu,x)$ generically on~$\cU\times \Proj$, not
only on~$\cV$.
$L(\bfu, x)$ is not defined when one of the~$\ud_x(f_i\circ u_i^{-1})$ is zero.
When it is defined, we can identify~$L(\bfu,x)$ with one of its preimages under the projection
map~$\bS(\bC^{n+1})^n \to (\Proj)^n$, that is a $n\times (n+1)$
matrix with unit rows. Namely,
\begin{equation}\label{eq:25}
  L(\bfu,x) = \mop{diag} \left( \|\ud_x(f_1\circ u_1^{-1})\|^{-1}, \dotsc,
    \|\ud_x(f_n\circ u_n^{-1})\|^{-1} \right) \cdot \ud_x \left( \bfu \cdot F
  \right).
\end{equation}
Under this identification, we check that~$L(\bfu, x)$ is the matrix of the
map
\[ \mop{proj}\left(\ker(\ud_x(f_1\circ u_1^{-1})^\perp, \dotsc, \ker(\ud_x(f_n\circ
    u_n^{-1})^\perp \right), \]
and therefore, by definition of~$\kappa$ \eqref{eq:38},
\begin{equation}\label{eq:10}
  \kappa(\bfu \cdot F, x) =
  \begin{cases}
    \tn{L(\bfu, x)^\dagger} & \text{if $L(\bfu, x)$ is well defined and surjective,}\\
    \infty & \text{otherwise.}
  \end{cases}
\end{equation}

Section~\ref{sec:path-lifting} studies the possibility of continuing a zero of a
polynomial system~$\bfu\cdot F$ when~$\bfu$ varies continuously along a path.
Section~\ref{sec:condition-number} introduces the condition number associated to
the solution variety~$\cV$: it quantifies the extend to which a zero of a
polynomial system~$\bfu \cdot F$ is affected by a pertubation of~$\bfu$.
Section~\ref{sec:cond-numb-their} proves some Lipschitz-continuity properties
for the condition number and~$\hat \gamma$. They are technical but crucial for designing
the continuation algorithms, what does Section~\ref{sec:yet-anoth-cont}.
Section~\ref{sec:randomized-algorithm} provides a bound on the average number of
continuation steps required to compute a zero of a random system~$\bfu \cdot F$
in the rigid setting. Theorem~\ref{thm:expectation-K} is the main outcome of
this part.

\subsection{Path lifting}
\label{sec:path-lifting}

Let~$\Sigma \subset \cV$ denote the \emph{singular locus}:
\[ \Sigma \eqdef \left\{ (\bfu, x)\in\cV \st \ud_x(\bfu \cdot F) \text{ is not surjective} \right\}. \]
In other words, $(\bfu, x) \in \Sigma$ if and only if~$\bfu \cdot F$ has a singular
zero at~$x$.
Let~$\pi : \cV \to \cU$ be the projection~$\pi(\bfu, x) = \bfu$
and~$\Sigma' \eqdef \pi(\Sigma)$.
For any~$\bfu \in \cU\setminus \Sigma'$, the polynomial system~$\bfu\cdot F$ has
finitely many roots that vary continuously with~$\bfu$.
Given a continous path~$(\bfw_t)_{0\leq t\leq 1}$ in~$\cU \setminus \Sigma'$ and
a zero~$\zeta$ of~$\bfw_0 \cdot F$, there is a unique continuous lifting~$(\bfw_t, \eta_t)
\in \cV$ such that~$\eta_0 = \zeta$.
In contrast with the classical setting, the parameter space~$\cU$ is not a
complex variety, so it is not obvious anymore that a generic path in~$\cU$ does
not meet~$\Sigma'$ and that this lifting is possible.
This section aims at proving this fact.

\begin{lemma}\label{prop:codim-sigma}
  The real codimension of~$\Sigma'$ in~$\cU$ is at least~$2$.
\end{lemma}

\begin{proof}
  We first observe that the map $\phi : \cU\times \Proj \to (\Proj)^n$ defined by
  \begin{equation*}
    (\bfu, x) \mapsto (u_1^{-1} x, \dotsc, u_n^{-1} x)
  \end{equation*}
  is a proper submersion, and thus a locally trivial fibration, by Ehresmann's
  fibration theorem.
  The fibers have dimension
  \begin{align*}
    \dim_\bR \phi^{-1}(*) &= \dim_\bR (\cU\times \Proj) - n\dim_\bR \Proj = \dim_\bR \cU - (n-1) \dim_\bR \Proj.
  \end{align*}
  The solution variety~$\cV$
  is~$\phi^{-1}(X_1\times\dotsb\times X_n)$,
  and so the restriction~$\phi_{|\cV}$ induces a locally trivial fibration~$\cV
  \to X_1\times\dotsb\times X_n$.
  In particular
  \begin{align*}
    \dim_\bR \cV &= \sum_{i=1}^n \dim_\bR X_i + \dim_\bR \phi^{-1}(*)\\
                 &= n(\dim_\bR \Proj - 2) + \dim_\bR \cU - (n-1)\dim_\bR \Proj\\
                 &= \dim_\bR \cU.
  \end{align*}
 
  Let~$\Sigma_0 \subseteq \Sigma$ be the subset of all~$(\bfu, x)$ such
  that~$x$ is a singular zero of one of the~$f_i\circ u_ i^{-1}$ (that is a
  singular point of~$u_i X_i$).
  By definition,
  \[ \Sigma_0 = \phi^{-1} \left( \bigcup_{i=1}^n  X_1\times\dotsb\times
      \mop{Sing} X_i \times \dotsb \times X_n \right), \]
  so~$\Sigma_0$ is the preimage by a locally trivial fibration of a complex
  subvariety of real codimension at least~$2$. Thus~$\Sigma_0$ has codimension
  at least~$2$ in~$\cV$.
  
  Let~$\Sigma_1 \subseteq \Sigma$ be the subset of all~$(\bfu, x)$ such that
  all~$u_i X_i$ are smooth at~$x$ but~$\ud_x(\bfu\cdot F)$ is not surjective,
  so that~$\Sigma = \Sigma_0 \cup \Sigma_1$.
  To compute the dimension of~$\Sigma_1$, we observe that the map
  \begin{align*}
    \psi : \cV\setminus \Sigma_0 &\to X_1^\text{reg}\times\dotsb\times X_n^\text{reg} \times \cL\\
    (\bfu, x) &\mapsto \left( u_1^{-1} x, \dotsc, u_n^{-1} x, L(\bfu, x) \right)
  \end{align*}
  is a proper submersion and thus,  by Ehresmann's fibration theorem,
  a locally trivial fibration.
  Further, $\Sigma_1$ is the preimage of the complex subvariety of
  all~$(y_1,\dotsc,y_n,\bfh)$ such that the intersection of the
  hyperplanes~$h_1,\dotsc,h_n \subset \bC^{n+1}$
  has complex dimension~$\geq 2$.
  So~$\Sigma_1$ has real codimension at least~$2$ in~$\cV$,
  and so does~$\Sigma$.

  Since~$\Sigma'$ is the image of~$\Sigma$ by the projection map~$\pi:\cV \to \cU$,
  $\dim_\bR \Sigma' \leq \dim_\bR \Sigma$, and since~$\dim_\bR \cU = \dim_\bR \cV$, it follows
  that $\Sigma'$ has codimension at least~$2$ in~$\cU$.
\end{proof}

\begin{proposition}\label{coro:possibility-of-continuation}
  For almost all~$\bfu, \bfv \in \cU$, the shortest path in~$\cU$ from~$\bfu$ to~$\bfv$
  does not intersect~$\Sigma'$.
\end{proposition}

\begin{proof}
  The subset~$S \subset \cU\times\cU$ of all ill-posed pairs $(\bfu, \bfv)$ that do not satisfy the claim is
  parametrized by the data of a~$\bfw \in \Sigma'$, a unit vector~$\dot \gamma$ in~$T_\bfw
  \cU$ and two real numbers~$t$ and~$s$ such that~$\bfu = \gamma(t)$ and~$\bfv =
  \gamma(s)$, where~$\gamma : \bR \to \cU$ is the unique geodesic with~$\gamma(0) = \bfw$
  and~$\gamma'(0) = \dot \gamma$.
  Therefore, by Lemma~\ref{prop:codim-sigma},
  \[ \dim_\bR S \leq \dim_\bR \Sigma' + (\dim_\bR \cU - 1) + 2 \leq \dim_\bR \cU^2 -
    1, \]
  and~$S$ has real codimension at least~$1$ in~$\cU \times \cU$, which proves the claim.
\end{proof}

The statement still holds for more general paths as long as they are unitary
invariant. Let~$P : \cU \times \cU \times [0,1] \to \cU$ be a map such that for
any~$\bfu, \bfv, \bfw \in \cU$ and~$t\in [0,1]$, $P(\bfw\bfu, \bfw\bfv, t) =
\bfw P(\bfu, \bfv, t)$.
Let~$S \subset \cU\times \cU$ be the set of all~$(\bfu, \bfv)$
such that~$P(\bfu, \bfv,
t)\in\Sigma'$ for some~$t\in [0,1]$.
We check easily that
\diffblock{\begin{equation*}
 S = \left\{ \left( \bfw P(\one_\cU, \bfv, t)^{-1}, \bfw P(\one_\cU, \bfv, t)^{-1} \bfv \right) \st t\in [0,1], \bfw \in \Sigma' \text{ and } \bfv \in \cU\right\},
\end{equation*}}
so if~$P$ is regular enough to allow dimension counting, then $\dim_\bR S \leq 1+ \dim_\bR \Sigma' + \dim_\bR
\cU$, and by Lemma~\ref{prop:codim-sigma}, the codimension of~$S$ in~$\cU\times
  \cU$ is at least~$1$.

\subsection{Condition number}
\label{sec:condition-number}

The rigid solution variety, considered as a manifold of pairs problem--solution
has a natural condition number. We show that this is the incidence condition number~$\kappa$,
defined in~\S \ref{sec:splitgamma}. This is what makes the split gamma
number~$\hat\gamma$ fit nicely in the setting of the rigid solution variety. The
system~$F$ being fixed, we will denote~$\kappa(\bfu \cdot F, x)$ and $\hat
\gamma(\bfu \cdot F, x)$ simply as~$\kappa(\bfu,x)$ and~$\hat\gamma(\bfu,x)$.


\begin{lemma}
  \label{lem:mu-as-kappa}
  For any~$(\bfu, x)\in \cV \setminus \Sigma$
  and any tangent vector~$(\dot\bfu,\dot x)\in T_{\bfu, x}\cV$,
  \[ \|\dot x\| \leq \kappa(\bfu, x)\|\dot \bfu\|_u. \]
\end{lemma}

\begin{proof}
  Without loss of generality, we may assume that~$\bfu = \one_\cU$
  which simplifies notations.
  The tangent space of~$\cV$ at~$(\bfu, x)$ is
  \begin{equation}\label{eq:13}
    \begin{aligned}[t]
      T_{\bfu, x} \cV
      &=
      \begin{multlined}[t][.7\linewidth]
        \left\{ (\dot\bfu, \dot x) \in T_\bfu\cU \times T_x \Proj \st \right . \\ \left. \forall i, f_i \left( (\one_\cU + t \dot u_i + o(t))^{-1} (x+t \dot x + o(t)) \right) = o(t) \text{ as } t\to 0 \right\}
      \end{multlined}\\
      &= \left\{ (\dot\bfu, \dot x) \in T_\bfu\cU \times T_x \Proj \st \forall i, f_i\left(x + t( \dot x - \dot u_i x) \right) = o(t) \text{ as } t\to 0 \right\}\\
      &= \left\{ (\dot\bfu, \dot x) \in T_\bfu\cU \times T_x \Proj
        \st \forall i, \ud_x f_i (\dot x) = \ud_x f_i( \dot u_i x) \right\}.
    \end{aligned}
  \end{equation}
  With the identification~\eqref{eq:25},
  we obtain that $(\dot\bfu,\dot x) \in
  T_{\bfu,x}\cV$ if and only if
  \begin{equation}\label{eq:26}
   L(\bfu, x)(\dot x) = \diffblock{\left( \|\ud_x f_i\|^{-1} \ud_x f_i( \dot u_i x)
    \right)_{1\leq i\leq n}}.
  \end{equation}
  For all~$1\leq i\leq n$, $\ud_xf_i(x) = \deg(f_i) f_i(x) =
  0$ (Euler's relation), therefore
  \begin{equation}\label{eq:42}
    |  \ud_xf_i( \dot u_i x )| = | \ud_xf_i( \pi_x(\dot u_i x) )| \leq \|\ud_x f_i\| \| \pi_x(\dot u_i x)\|,
  \end{equation}
  where~$\pi_x$ is the orthogonal projection on~$\left\{ x
  \right\}^\perp$. We check that $\| \pi_x(\dot u_i x)\| \leq \|\dot u_i\|_u$,
  as in Lemma~\ref{lem:riem-submersion-proj}.
  If Equation~\eqref{eq:26} does hold, then
  \begin{align*}
    \| L(\bfu, x)(\dot x) \|^2 &=  \sum_i \| \ud_x f_i \|^{-2} | \ud_xf_i( \dot u_i x )|^2\\
                               &\leq \sum_i \|\dot u_i\|_u^2, && \text{by \eqref{eq:42},}\\
                               &= \|\dot\bfu\|^2_u.
  \end{align*}
  Moreover~$\dot x = L(\bfu, x)^\dagger ( L(\bfu, x)(\dot x))$, because~$\dot x$ is orthogonal
  to~$x$ and the kernel of~$L(\bfu, x)$ is~$\bC x$.
  Therefore
  \[ \|\dot x\| \leq \tn{L(\bfu, x)^{\dagger}} \|L(\bfu, x)(\dot x) \| \leq
    \tn{L(\bfu, x)^{\dagger}} \|\dot \bfu\|_u, \]
  and the claim follows with~$\kappa(\bfu, x) =  \tn{L(\bfu, x)^{\dagger}}$, by~\eqref{eq:10}.
\end{proof}

The expected value of the incidence condition number~$\kappa$ is very tame and depends on
the ambient dimension~$n$ only.
This is one of the key points that strongly contrasts with the
classical setting. 

\begin{proposition}\label{prop:expectation-kappa}
  If $(\bfu,\zeta)\in \cV$ is~$\rstd$-distributed then,
  $\bE\left[\kappa(\bfu,\zeta)^2\right] \leq 6 n^2$.
\end{proposition}

\begin{proof}
  Let~$M$ be a random $n\times (n+1)$ matrix whose rows are independent and
  uniformly distributed in~$\bS(\bC^{n+1})$. It follows from~\eqref{eq:10} and
  Theorem~\ref{thm:linearization}(i) that $\kappa(\bfu,\zeta)$ has the same
  probability distribution as~$\VERT M^\dagger \VERT$.

  Let~$T$ be an $n\times n$ diagonal random matrix whose coefficients are independent
  $\chi$-distributed random variables with $2n+2$ degrees of freedom, so that~$TM$
  is a random Gaussian matrix (the coefficients are independent standard normal
  complex numbers). Obviously, $M^\dagger = (TM)^\dagger \cdot T$ and therefore
  $\VERT{M^\dagger}\VERT \leq \VERT(TM)^\dagger\VERT
  \tn{T}$. Hölder's inequality with conjugate exponents~$n' \eqdef
  1+\frac{1}{n+1}$ and~$n+2$, gives
  \begin{equation}
    \label{eq:23}
    \bE \left[ \VERT M^\dagger\VERT^2 \right]\leq \bE \left[ \VERT(TM)^\dagger\VERT^{2n'} \right]^\frac{1}{n'} \bE \left[ \tn{T}^{2n+4} \right]^\frac{1}{n+2}. 
  \end{equation}
  We now give upper bounds for both factors in the right-hand side.
  According to
  \textcite[Theorem~20]{BeltranPardo_2011}, whose result is derived from the
  work of \textcite{Edelman_1989},
  \begin{equation}\label{eq:9}
    \bE \left[ \VERT(TM)^\dagger\VERT^{2n'} \right] = 2^{-n'} \sum_{k=1}^{n} \binom{n+1}{k+1} \frac{\Gamma(k-n'+1)}{\Gamma(k)} n^{-k+n'-1}.
  \end{equation}
  We proceed as in \parencite[Theorem~10]{Lairez_2017}
  and deduce that
  \begin{equation}
    \bE \left[ \VERT(TM)^\dagger\VERT^{2n'} \right]^\frac{1}{n'} \leq  \left( \frac{5}{4 (2-n')} \right)^{\frac{1}{n'}} \frac{n}{2} \leq n.
  \end{equation} 
  Concerning the second factor, 
  $\tn{T}^2$ is the maximum of $n$
  $\chi^2$-distributed random variables with $2n+2$ degrees of freedom, say
  $Z_1,\dotsc,Z_n$, hence
  \[ \bE \left[ \tn{T}^{2n+4} \right] \leq \sum_{i=1}^n \bE[Z_i^{n+2}] = 2^{n+2}
    \frac{(2n+2)!}{(n-1)!}. \]
  Therefore,
  \[ \bE \left[ \VERT M^\dagger\VERT^2 \right] \leq  n \left( 
         2^{n+2}
    \frac{(2n+2)!}{(n-1)!} \right)^{\frac{1}{n+2}} \leq 6 n^2, \]
  after a few numerical computations.
\end{proof}

\subsection{Lipschitz properties}
\label{sec:cond-numb-their}

We aim at bounding the variation of the numbers~$\kappa$
and~$\hat\gamma$ on the Riemannian manifold~$\cU\times \Proj$.
In particular, we will prove that~$1/\hat \gamma$ is Lipschitz-continuous.
Traditionally, such results
bound directly the value at some point~$x$ with respect to the value at some other
point~$y$ and the distance from~$x$ to~$y$. For example
\parencite[Lemme~131]{Dedieu_2006}, for any~$x,y\in\Proj$,
\begin{equation}
  \label{eq:3}
  \gamma(F,y) \leq \gamma(F,x) \, q\left( d_\bP(x,y) \gamma(F,x) \right),
\end{equation}
where~$q(u) \eqdef \frac{1}{(1-u)(1-4u+2u^2)} = 1 + 5u + O(u^2)$, given
that~$d_\bP(x,y) \gamma(F,x) \leq 1-1/\sqrt{2}$.
As much as possible, I tried to express this kind of inequalities
as a bound on the derivative of the function under consideration.
To first order, this is equivalent.

\begin{proposition}
  \label{prop:gamma-lip}
  Let~$F\in\cH[r]$ and~$\gamma_F : x\in \Proj \mapsto \gamma(F,x)$.
  For any~$x\in\Proj$, we have $\|\ud_x \gamma_F\| \leq 5 \gamma_F^2$.
\end{proposition}

Note that $\gamma_F$ may not be differentiable everywhere, so the inequality $\|\ud_x
\gamma_F\| \leq 5 \gamma_F(x)^2$
really means that
\[ \limsup_{y\to x} \frac{\left|\gamma_F(y)-\gamma_F(x)\right|}{d_\bP(y,x)} \leq
  5 \gamma_F(x)^2. \]
It is easy to check that Proposition~\ref{prop:gamma-lip} is equivalent to the
Lipschitz continuity of the function~$1/\gamma_F$, with Lipschitz constant at most~$5$.
We give~$\|\ud \kappa\|$ and~$\|\ud \hat\gamma\|$ an analogue meaning.

\begin{proof}[Proof of Proposition~\ref{prop:gamma-lip}]
  The most direct way to see this is by~\eqref{eq:3}:
  \[ \frac{\gamma_F(y)-\gamma_F(x)}{d_\bP(y,x)} \leq \gamma_F(x)
    \frac{5 u + O(u^2)}{d_\bP(y,x)} = 5 \gamma_F(x)^2 \left( 1 + O(d_\bP(x,y))
    \right), \]
  as~$y\to x$, where~$u\eqdef d_\bP(x,y) \gamma_F(x)$,
  and
  \[ \frac{\gamma_F(x)-\gamma_F(y)}{d_\bP(x,y)} \leq \gamma_F(y)
    \frac{5 v + O(v^2)}{d_\bP(x,y)} = 5 (\gamma_F(x) + o(1))^2 \left( 1 + O(d_\bP(x,y)
    \right), \]
  where~$v\eqdef d_\bP(x,y) \gamma_F(y)$.
\end{proof}

\begin{lemma}
  \label{lemma:kappa-lip}
  On~$\cU\times \Proj$,
  $\| \ud\kappa \| \leq \kappa^2  + 3 \kappa \hat\gamma$.
  Moreover, if~$D \geq 2$ then $\|\ud\kappa\| \leq 5 \kappa\hat\gamma$.
\end{lemma}

\begin{proof}
  The second inequality follows from the first one: If~$D \geq 2$ then at least
  one~$\gamma(u_i \cdot f, x)$ is greater or equal to~$\frac{D-1}{2}$, by
  Lemma~\ref{lem:gamma-lower-bound}. It follows that~$\kappa \leq 2 \hat\gamma$
  and then~$\kappa^2 + 3 \kappa \hat\gamma \leq 5 \kappa\hat\gamma$.
  
  To prove the first inequality, we first remark that~$1/\kappa(\bfu,x)$ is a
  Lipschitz-continuous function of~$L(\bfu, x)$ with constant~$1$. Indeed,
  $1/\kappa(\bfu,x)$ is the least singular value of~$L(\bfu, x)$ as a matrix, see
  Equation~\eqref{eq:10}, and the Eckart--Young theorem expresses this number as the
  distance to the set of singular matrices, which is a Lipschitz continuous
  function with constant~$1$. Moreover~$\ud \kappa = - \kappa^2 \ud
  \frac{1}{\kappa}$, so it is enough to prove that~$\tn{\ud L}$ is bounded by~$1
  + 3 \frac{\hat\gamma}{\kappa}$.
  
  Let~$L_i(u_i,z) \in \Proj$ be the $i$th component of~$L(\bfu,z)$, that is the
  projective class of the linear form~$\ud_x(u_i\cdot f_i)$.
  The tangent space of~$\Proj$ at~$L_i(u_i, z)$ is isometrically identified with the
  quotient~$\bC^{n+1}/ \bC \cdot L_i(u_i,z)$.
  Denoting~$h_i \eqdef u_i\cdot f_i$, we check that, at a point~$(\mathbf u, x)$,
  \[ \ud L_i(0, \dot x) = \frac{1}{\|\ud_x h_i\|} \ud_x^2 h_i(\dot x)
    \mod{L_i(u_i,z)},\]
  and in particular, $\|\ud L_i(0, \dot x)\| \leq 2 \gamma(h_i,x) \|\dot x\|$
  for any~$\dot x \in T_x\Proj$. Besides, $L_i(u_i,u_ix) = L_i(\mop{id},x)\circ
  u_i^*$, which is a $1$-Lipschitz continuous function of~$u_i$
  (Lemma~\ref{lem:riem-submersion-proj} applied to the dual projective space).
  This proves that $\|\ud L_i(\dot u_i, \dot u_i x)\| \leq \|\dot u_i\|_u$.
  Therefore,
  \begin{align*}
    \| \ud L_i(\dot u_i, \dot x) \|
    &\leq \|\ud L_i(\dot u_i,\dot u_i x) \| + \|\ud L_i(0, \dot x - \dot u_i x) \|\\
    &\leq \|\dot u_i\|_u + 2\gamma(h_i,x) \left( \|\dot x\| + \|\dot u_i\|_u \right)\\
    &\leq \|\dot u_i\|_u + 2\sqrt{2} \gamma(h_i,x) \left( \|\dot \bfu\|_u^2 + \|\dot x\|^2 \right)^{\frac12},
  \end{align*}
  and then, by the triangle inequality,
  \begin{align*}
    \left\| \ud L(\dot\bfu,\dot x) \right \|
    & \leq \|\dot \bfu\|_u + 2\sqrt{2} \bigg(\sum_i \gamma(h_i,x)^2 \bigg)^{\frac{1}{2}} \left( \|\dot \bfu\|_u^2 + \|\dot x\|^2 \right)^{\frac12}\\
    & \leq \left(1+3 \frac{\hat\gamma(\bfu;x)}{\kappa(\bfu,x)} \right) \left( \|\dot \bfu\|_u^2 + \|\dot x\|^2 \right)^{\frac12},
  \end{align*}
  which concludes the proof.
\end{proof}

We now derive a bound for~$\|\ud \hat \gamma\|$.




\begin{lemma}
  \label{prop:gamma-single-lip}
  For any homonegeneous polynomial~$f:\bC^{n+1} \to \bC$
  the map
  \[ (u,x)\in U(n+1)\times\Proj \longmapsto \frac{1}{\gamma(u\cdot f, x)} \]
  is Lipschitz continuous with constant at most~$5\sqrt{2}$.
\end{lemma}

\begin{proof}
  The $\gamma$ number is invariant under unitary transformations, that is
  $\gamma(u\cdot f, x) = \gamma(f, u^* x)$. Moreover, the map~$(u,x)\in
  U(n+1)\times \Proj\mapsto u^* x\in \Proj$ is $1$-Lipschitz continuous with respect
  to~$u$ (Lemma~\ref{lem:riem-submersion-proj}) and to~$x$, thus it
  is~$\sqrt{2}$-Lipschitz continuous on~$U(n+1)\times \Proj$.
  Since~$1/\gamma(f,x)$ is $5$-Lipschitz continous with respect to~$x$
  (Proposition~\ref{prop:gamma-lip}), the map~$1/\gamma(f, u^* x)$ is
  $5\sqrt{2}$-Lipschitz continuous on~$U(n+1)\times \Proj$.
\end{proof}

\begin{lemma}
  \label{lemma:gammas-lip}
    On $\cU\times \Proj$, $\left\| \ud\hat\gamma\right\| \leq 13\, \hat\gamma^2$.
    Equivalently, $1/\hat\gamma$ is 13-Lipschitz continuous.
\end{lemma}

\begin{proof}
  We may assume that~$D \geq 2$, otherwise~$\hat\gamma$ is identically~$0$.
  Let~$\gamma_i(\bfu, x) \eqdef \gamma(u_i\cdot f_i, x)$
  and~$\eta \eqdef \sum_i \gamma_i^2$, so that~$\hat\gamma = \kappa \sqrt{\eta}$.
  By Lemma~\ref{prop:gamma-single-lip}, $\|\ud\gamma_i\| \leq 5 \sqrt{2}
  \gamma_i^2$ and then
  \[ \|\ud \eta\| \leq 2 \sum_i \gamma_i \|\ud \gamma_i\| \leq 10 \sqrt{2} \sum_i \gamma_i^3 \leq 10 \sqrt{2} \eta^{3/2}. \]
  Therefore,
  \begin{align*}
    \left\|\ud\hat\gamma\right\| &\leq \|\ud \kappa \| \sqrt{\eta} + \tfrac12 \kappa \eta^{-\frac12} \|\ud \eta\| \\
                                 &\leq 5 \hat\gamma^2 + 5 \sqrt{2} \hat\gamma \sqrt{\eta},
                                 && \text{by Lemma~\ref{lemma:kappa-lip},}\\
                                 &\leq 13 \hat\gamma^2, &&\text{using~$\kappa \geq 1$, Lemma~\ref{lem:kappa-geq-1}.} \qedhere
  \end{align*}
\end{proof}



\subsection{Numerical continuation along rigid paths}
\label{sec:yet-anoth-cont}

We describe a continuation algorithm in the rigid solution variety and bound its
complexity in terms of the integral of~$\kappa \hat \gamma$ along the
continuation path. It is the analogue of the $\mu^2$ estimate of the classical
theory, see~\S \ref{sec:state-art}. The approach proposed here differs
from the usual treatment only in a more systematic use of derivatives.
We assume~$D \geq 2$ as otherwise there is only a
linear system of equations to solve.

As we will see in \S \ref{sec:aver-compl-dense},
it may be valuable not to compute $\hat \gamma$ but rather an
easier to compute upper bound. That is why the algorithm is described in terms of a
function~$g : \cU \times \Proj \to (0,\infty]$ that can be chosen freely, as long as:
\begin{enumerate}[(H1)]
\item \label{item:1} $\hat\gamma \leq g$, on $\cU
  \times \Proj$;
\item \label{item:2} $\frac{1}{g}$ is $C$-Lipschitz continuous, for
  some~$C \geq 10$.
\end{enumerate}

The following proposition describes one continuation step. Observe that the
conclusion~(ii) concerning the triple~$(\bfv, \zeta', z')$ is similar to the hypothesis~(a)
concerning~$(\bfu, \zeta, z)$, so that we can chain the continuation steps.
The geodesic distance in~$\cU$ is denoted~$d_\cU$.

\begin{proposition}\label{prop:homotopy-step}
  Let~$\bfu,\bfv\in \cU$, let $\zeta$ be a zero of the polynomial system~$\bfu\cdot F$, let~$z\in\Proj$ and let
  $z' \eqdef \cN_{\bfv\cdot F}(z)$. For any positive real number~$A\leq \frac{1}{4 C}$, if
  \begin{enumerate}[(a)]
  \item $d_\bP(z,\zeta) g(\bfu,\zeta) \leq A$, and
  \item $d_\cU(\bfu, \bfv) \kappa(\bfu, z) g(\bfu, z) \leq \frac 1{4} A$,
  \end{enumerate} 
  then there exists a unique zero~$\zeta'$ of\/~$\bfv \cdot F$ such that
  \begin{enumerate}[(i)]
  \item $z$ is an approximate zero of~$\bfv \cdot F$ with associated zero~$\zeta'$, and
  \item $d_\bP(z',\zeta') g(\bfv, \zeta') \leq A$.
  \end{enumerate}
\end{proposition}

\begin{proof}
  It suffices to construct a zero~$\zeta'$ of~$\bfv\cdot F$ such that~$d_\bP(z, \zeta') g(\bfv, \zeta') \leq 2A$.
  Indeed, since~$2A \leq \frac16$ and $\gamma\leq \hat\gamma \leq g$, it would imply by
  Theorem~\ref{thm:gamma} that: (i)~$z$ is an approximate zero of~$\bfv \cdot F$; and~(ii)
  that
  \[ d_\bP(z',\zeta')g(\bfv, \zeta') \leq \tfrac12 d_\bP(z,\zeta')g(\bfv, \zeta')
    \leq A. \]
  
  Consider a geodesic~$t \in [0,\infty) \mapsto \bfv_t\in\cU$ such
  that~$\|\dot \bfv_t\|=1$, $\bfv_0 = \bfu$ and $\bfv_{d_\cU(\bfu, \bfv)} =
  \bfv$.
  We may assume that~$\kappa(\bfu, \zeta) < \infty$, otherwise~$\bfu = \bfv$, by
  Hypothesis~(b),
  and there is nothing to prove.
  So~$(\bfu, \zeta)$ is not in the singular locus~$\Sigma$ and 
  for~$t$ small enough, there is a unique lift~$(\bfv_t, \eta_t)$ in~$\cV
  \setminus \Sigma$, see~\S \ref{sec:path-lifting}.
  Let~$[0, \tau)$,
  with~$0 < \tau \leq \infty$, be the maximal interval of definition of~$\eta_t$.
  
  For~$t \in [0,\tau)$, let~$p_t
  \eqdef (\bfv_t,\eta_t)$ and let~$\delta_t \eqdef d_\bP(z,\eta_t)$. Let
  moreover $g_t \eqdef g(\bfv_t, \eta_t)$, $\beta_t \eqdef g_t
  \delta_t$ and~$\kappa_t \eqdef \kappa(\bfv_t, \eta_t)$. The derivative with
  respect to~$t$ is denoted with a dot.
  
  We first observe that~$\dot\delta_t \leq \kappa_t$, by
  Lemma~\ref{lem:mu-as-kappa}, and that
  \begin{equation}\label{eq:30}
    \|\dot p_t\| \leq (1 + \dot\delta_t^2)^{\frac12} \leq 2 \kappa_t
  \end{equation}
  (using~$\kappa \geq 1$, by Lemma~\ref{lem:kappa-geq-1}). By
  Lemma~\ref{lemma:kappa-lip}, $\dot\kappa_t \leq 5 \kappa_t g_t \|\dot p_t\| \leq 10
  \kappa_t^2g_t$ and by the Lipschitz hypothesis on~$1/g$, we have $\dot g_t \leq C g_t^2 \| \dot
  p_t\| \leq 2C \kappa_t g_t^2$. This implies that $\tfrac{\ud}{\ud t} \kappa_t g_t \leq
  3 C ( \kappa_t g_t )^2$ (using~$C \geq 10$) and equivalently,
  \begin{equation}\label{eq:43}
    \frac{\ud}{\ud t} \frac{1}{\kappa_t g_t} \geq -3 C.
  \end{equation}
  It follows, after integration, that $\frac{1}{\kappa_t g_t} \geq
  \frac{1}{\kappa_0 g_0} - 3Ct$, and thus
  \begin{equation}\label{eq:7}
    \kappa_t g_t \leq \frac{\kappa_0 g_0}{1-3 C t \kappa_0 g_0},
  \end{equation}
  for any~$t \in [0, \tau)$ such that~$1-3 Ct\kappa_0g_0 > 0$.
  
  A first consequence of~\eqref{eq:7} is that
  \begin{equation}\label{eq:31}
    \tau \geq (3C\kappa_0 g_0)^{-1}.
  \end{equation}
  Indeed, assuming that $\tau < (3C\kappa_0 g_0)^{-1}$,
  Inequalities~\eqref{eq:30}, \eqref{eq:7} and $g_t \geq \frac12$
  (Lemma~\ref{lem:gamma-lower-bound}) show that~$\|\dot p_t\|$ is bounded
  on~$[0, \tau)$. Therefore~$p_t$ has a limit as~$t \to \tau$,
  and then $\eta_t$ is defined on the interval~$[0, \tau]$.
  But~$g_\tau < \infty$, by~\eqref{eq:7} with~$t\to\tau$, therefore $\eta_\tau$
  is not a singular zero of~$\bfv_\tau \cdot F$ and~$\eta_t$ can be continued
  in a neighborhood of~$\tau$ contradicting that~$\tau$ is maximal.
  
  Next, we compute that
  \begin{align*}
    \dot\beta_t &= g_t \dot\delta_t + \dot g_t \delta_t\\
                &\leq g_t \kappa_t + 2C \kappa_t g_t^2 \delta_t \\
                &= (1 + 2C \beta_t) \kappa_t g_t,
  \end{align*}
  then
  \begin{align*}
    \frac{\ud}{\ud t} \log(1 + 2C \beta_t) &\leq 2 C \kappa_t g_t \\
                                           &\leq \frac{2 C \kappa_0 g_0}{1-3 C t \kappa_0 g_0}, && \text{by \eqref{eq:7},} \\
                                           &= - \frac23 \frac{\ud}{\ud t} \log(1-3 Ct \kappa_0g_0),
  \end{align*}
  and therefore, after
  integration with respect to~$t$,
  \begin{equation}\label{eq:28}
    \log\left( \frac{1+ 2C\beta_t}{1+2C\beta_0}\right)
    \leq  \frac23 \log\left( \frac{1}{1-3C t \kappa_0g_0} \right).
  \end{equation}
  Exponentiating both sides leads to
  \begin{equation}\label{eq:5}
    \beta_t \leq \frac{1 + 2C \beta_0}{2C\left( 1-3C t\kappa_0g_0 \right)^{\frac23}} - \frac{1}{2C}.
  \end{equation}

  We now bound~$\kappa_0 g_0$. As a function on~$\cU\times \Proj$, we compute
  that
  \begin{equation}
    \begin{aligned}[t]
      \|\ud (\kappa g)\| &\leq \|\ud\kappa\| g + \|\ud g\| \kappa \\
      &\leq 5\kappa \hat \gamma g + C g^2 \kappa, \text{ by Lemma~\ref{lemma:kappa-lip} and $C$-Lipschitz continuity of $\frac1g$,} \\
      &\leq \tfrac32 C \kappa g^2, \quad \text{because $\hat \gamma \leq g$ and~$10 \leq C$, by assumption,} 
    \end{aligned}\label{eq:16}
  \end{equation}
  and then,
  \begin{equation} \label{eq:27}
    \|\ud \log(\kappa g)\| = (\kappa g)^{-1} \|\ud (\kappa g)\| \leq \tfrac32 C g.
  \end{equation}
  Since~$1/g$ is $C$-Lipschitz continuous on~$\cU\times\Proj$, for any~$w \in \Proj$
  on a shortest path between~$z$ and~$\zeta$, $|g(\bfu, w)^{-1} - g(\bfu,
  \zeta)^{-1}| \leq C d_\bP(\zeta,z)$ and then
  \begin{equation}\label{eq:17}
    g(\bfu,w) \leq \frac{g(\bfu,\zeta)}{1-C d_\bP(\zeta,z) g(\bfu,\zeta)}.
  \end{equation}
  After integrating the relation \eqref{eq:27} on a path
  from~$(\bfu, \zeta)$ to~$(\bfu, z)$, and bounding
  the right-hand side with~\eqref{eq:17}, we obtain
  \begin{equation*}
    \log(\kappa_0 g_0) \leq \log( \kappa(\bfu, z) g(\bfu, z) ) + \frac{\tfrac32 Cd_\bP(\zeta,z) g(\bfu,\zeta)}{1-C d_\bP(\zeta,z) g(\bfu,\zeta)},
  \end{equation*}
  and then
  \begin{equation}\label{eq:18}
    \kappa_0g_0 \leq \kappa(\bfu,z) g(\bfu,z)
    \exp \left( \frac{\tfrac32 C d_\bP(z,\zeta) g(\bfu,\zeta)}{1-C d_\bP(z,\zeta) g(\bfu,\zeta)} \right).
  \end{equation} 
  We multiply by~$d_\cU(\bfu,\bfv)$ both sides, use the hypotheses~(a) and~(b), and
  obtain
  \begin{equation}\label{eq:29}
    d_\cU(\bfu,\bfv) \kappa_0 g_0 \leq \tfrac14 A \exp \left( \frac{\frac32 C A}{1- C A} \right) < \tfrac{5}{12} A,
  \end{equation}
  where the last inequality follows from the hypothesis $CA \leq \frac{1}{4}$.
  Together with~\eqref{eq:31} and the hypothesis $CA \leq \frac{1}{4}$, we
  deduce $d_\cU(\bfu, \bfv) < \tau$.
  In particular, we can define~$\zeta' = \eta_{d_\cU(\bfu, \bfv)}$, which is a
  zero of~$\bfv\cdot F$, and apply
  Inequality~\eqref{eq:5} at~$t = d_\cU(\bfu,\bfv)$. This gives, in combination
  with~\eqref{eq:29} and Hypothesis~(a),
  \begin{align*}
    d_\bP(z, \zeta') g(\bfv, \zeta')
    &\leq \frac{1 + 2CA}{2C ( 1- \frac54 CA)^{\frac23}} - \frac{1}{2C} \\
    &\leq \left( \frac{1 + \tfrac{2}{4}}{\tfrac{2}{4} (1-\frac{5}{16})^{\frac23}} - \frac{4}{2} \right) A \\
    &\leq 2 A
  \end{align*}
  using again that~$C A\leq \frac1{4}$.
  This concludes the proof.
\end{proof}

\begin{algo}[tp]
  \centering
  \begin{algorithmic}
    \Function{NC}{$F$, $\bfu$, $\bfv$, $z$}
    \State $(\bfw_t)_{0\leq t \leq T} \gets$ a 1-Lipschitz continuous path
    from~$\bfv$ to~$\bfu$
    \State $t\gets {1}/\left(16 C\, \kappa(\bfw_0, z) g(\bfw_0, z)\right)$
    \While{$t < T$}
    \State $z \gets \cN_{\bfw_t}(z)$
    \State $t \gets t + {1}/\left(16 C \, \kappa(\bfw_t, z) g(\bfw_t, z)\right)$        
    \EndWhile
    \State \textbf{return} $z$
    \EndFunction
  \end{algorithmic}
  \caption[]{Numerical continuation
    \begin{description}
      \item[Input.] $F \in\cH[n]$, $\bfu$, $\bfv \in \cU$ and~$z \in \Proj$
      \item[Precondition.] $z$ is a zero of~$\bfv \cdot F$ and the function~$g$
        satisfies~\ref{item:1} and~\ref{item:2}.
      \item[Output.] $w\in\bP^n$.
      \item[Postcondition.] $w$ is an approximate zero of~$\bfu \cdot F$.
    \end{description}
  }
  \label{algo:hc}
\end{algo}

Based on Proposition~\ref{prop:homotopy-step}, the procedure {NC} (Algorithm~\ref{algo:hc}) computes
an approximate zero of a system~$\bfu \cdot F$ given a zero of another
system~$\bfv \cdot F$ using a numerical continuation along a
path~$(\bfw_t)_{0 \leq t \leq T}$ from~$\bfv$ to~$\bfu$.

\begin{theorem}\label{thm:complexity-numcont}
  On input~$F$, $\bfu$, $\bfv$ and~$z$,
  assuming that~$z$ is a zero of\/~$\bfv \cdot F$,
  Algorithm~\ref{algo:hc} ouputs an approximate zero of\/~$\bfu \cdot F$ or loops forever.
  
  If the continuation path~$(\bfw_t)_{0\leq t\leq T}$ 
  chosen by the algorithm lifts as a continous path~$(\bfw_t, \zeta_t)$ in~$\cV$
  with~$\zeta_0 = z$, then
  the the algorithm terminates after at most
  \[ 25 C \int_0^T \kappa(\bfw_t,\zeta_t) g(\bfw_t, \zeta_t) \ud t \]
  continuation steps.
\end{theorem}

\begin{proof}
  Let~$t_0 \eqdef 0$, let~$t_k$ be the value of~$t$ at the beginning of the~$k$th
  iteration, let~$z_0\eqdef z$ and let~$z_k$ be the value of~$z$ at the end of the~$k$th iteration;
  namely
  \begin{align}
    t_{k+1} &\eqdef t_k + \frac{1}{16 C\, \kappa(\bfw_{t_k}, z_k) g(\bfw_{t_k}, z_k)} \label{eq:12}\\
    z_{k+1} &\eqdef \cN_{\bfw_{t_{k+1}}}(z_{k}).
  \end{align}
  Let~$K$ be the largest integer such that~$t_K \leq T$ (or $K\eqdef \infty$ if
  there is no such integer).
  The output of the algorithm, if any, is~$z_K$.
  Thanks to the Lipschitz hypothesis for~$\bfw$, we have, for any~$k \geq 0$,
  \[ d_\cU(\bfw_{t_k}, \bfw_{t_{k+1}}) \leq t_{k+1}-t_k = \left( 16 C
      \kappa(\bfw_{t_k},z_k)g(\bfw_{t_k},z_k) \right)^{-1}. \]
  Repeated application of Proposition~\ref{prop:homotopy-step} with~$A\eqdef
  \frac{1}{4 C}$ leads to
  \begin{equation}\label{eq:44}
    d_\bP(z_k,\zeta_{t_k})
    g(\bfw_{t_k},\zeta_{t_k}) \leq A = \frac1{4 C},
  \end{equation}
  for any~$k \geq 0$ (conclusion (ii) of
  Proposition~\ref{prop:homotopy-step} is used for hypothesis (a) at the next
  step, initialization is trivial since~$z$ is a zero). By
  Proposition~\ref{prop:homotopy-step} again, for any~$k \geq 0$ and any $t\in
  [t_k,t_{k+1}]$, $z_k$ is an approximate zero of~$\bfw_t \cdot F$. In particular, $z_K$
  is an approximate zero of~$\bfw_T \cdot F$. This proves the correctness of the
  algorithm.

  Concerning the bound on the number of iterations, we first note that
  \begin{align}
    \int_0^T \kappa_t g_t \ud t
    &\geq \int_0^{t_K} \kappa_t g_t \ud t,  && \text{because~$t_K < T$,}\notag\\
    &\geq \sum_{k=0}^{K-1} (t_{k+1} - t_k) \min_{t_k \leq s < t_{k+1}} \kappa_s g_s \notag \\
    &= \sum_{k=0}^{K-1} \frac{\min_{t_k \leq s < t_{k+1}} \kappa_s g_s}{16 C\, \kappa(\bfw_{t_k}, z_k) g(\bfw_{t_k}, z_k)}, \label{eq:19}
  \end{align}
  where we use the notations of the proof of Proposition~\ref{prop:homotopy-step}
  applied to the path~$(\bfw_t)$.

  Similarly to~\eqref{eq:43}, but aiming now for
  lower bounds, we compute that
  \[ \frac{\ud}{\ud s} \frac{1}{\kappa_s g_s} \leq 3C. \]
  After integration,
  analogously to~\eqref{eq:7}, we obtain that for any~$t_k \leq s < t_{k+1}$
  \begin{equation*}
    \kappa_s g_s \geq \frac{\kappa_{t_k} g_{t_k}}{1+3 C (t_{k+1}-t_k) \kappa_{t_k} g_{t_k}}.
  \end{equation*}
  We also check, similarly to~\eqref{eq:18}, integrating along a shortest path
  from~$\zeta_{t_k}$ to~$z_k$, that
  \begin{align*}
    \kappa_{t_k} g_{t_k} &\geq  \kappa(\bfw_{t_k},z_{k}) g(\bfw_{t_k},z_k)
                           \exp \left( - \frac{\frac32 C d_\bP(z_k,\zeta_{t_k}) g(\bfw_{t_k},\zeta_k)}{1+C d_\bP(z_k,\zeta_{t_k}) g(\bfw_{t_k},\zeta_{t_k})} \right) \\
      &\geq \kappa(\bfw_{t_k},z_{k}) g(\bfw_{t_k},z_k) \exp( -\tfrac{3}{10} ), \quad\text{with~\eqref{eq:44}.}
  \end{align*}
  Therefore, for any~$t_k \leq s < t_{k+1}$,
  \begin{align*}
    \kappa_s g_s &\geq \frac{\kappa(\bfw_{t_k},z_{k}) g(\bfw_{t_k},z_k) \exp( -\tfrac{3}{10} )}{1+3 C (t_{k+1}-t_k) \kappa(\bfw_{t_k},z_{k}) g(\bfw_{t_k},z_k) \exp( -\tfrac{3}{10} )} \\
    &\geq \frac{\kappa(\bfw_{t_k},z_{k}) g(\bfw_{t_k},z_k)}{\exp \left( \tfrac3{10} \right) + \tfrac3{16}} , \text{ using the value for $t_{k+1}-t_k$ \eqref{eq:12}},
  \end{align*}
  and then
  \begin{equation*}
     \frac{\min_{t_k \leq s < t_{k+1}} \kappa_s g_s}{16 C\, \kappa(\bfw_{t_k}, z_k) g(\bfw_{t_k}, z_k)} \geq \frac{1}{16 \exp \left( \tfrac3{10} \right) + 3}\cdot \frac{1}{C} \leq \frac{1}{25 C}.
  \end{equation*}
  Therefore, by \eqref{eq:19}, $\int_0^T \kappa_tg_t\ud t \geq \frac{1}{25 C} K$.
\end{proof}

We have some degrees of freedom but also some constraints in the choice of the
path $(\bfw_t)_{0\leq t\leq T}$ from~$\bfv$ to~$\bfu$:
\begin{enumerate}[(P1)]
\item\label{item:9} the path is~$1$-Lipschitz continuous;
\item\label{item:5} the path~$(\bfv^{-1} \bfw_t)_{0\leq t\leq T}$ (from~$\one_\cU$
  to~$\bfv^{-1} \bfu$) depends only on~$\bfv^{-1} \bfu$; and
\item\label{item:6} the length $T$ of the path is at most~$4n$.
\end{enumerate}
The first one is required by the numerical continuation algorithm.
The two others will be useful for the complexity analysis.

An obvious choice of the path between~$\bfv$ and~$\bfu$ is the shortest
one: we write $\bfv^{-1} \bfu = \left(\exp(A_1),\dotsc,\exp(A_n) \right)$ for
some skew-Hermitian matrices~$A_1,\dotsc,A_n$ and define
\[ \bfw_t \eqdef \bfv \left( \exp( \tfrac t T A_1 ), \dotsc, \exp( \tfrac tT
    A_n) \right), \text{ for $0\leq t\leq T$,} \]
where~$T \eqdef (\sum_i \|A_i\|^2_u)^{\frac12}$. We can always choose the
matrices~$A_i$ such that $T \leq 4n$ as follows:
diagonalize each~$A_i$ in a unitary basis (which preserves the Frobenius
norm) as~$\mop{diag}(\theta_{i,0},\dotsc,\theta_{i,n}) \sqrt{-1}$, for some~$\theta_{i,j}
\in \bR$; add integer multiples of~$2\pi$ to the $\theta_{i,j}$ (which
preserves the end points of the path) so that~$\theta_{i,j} \in [-\pi,\pi]$;
and finally, obtain
\[ T^2=\sum_{i=1}^n \|A_i\|^2_u = \frac{1}{2}\sum_{i = 1}^n \sum_{j=0}^n
  \theta_{i,j}^2 \leq \frac12 n(n+1) \pi^2 < (4n)^2. \]
Naturally, it may
not be convenient to compute matrix logarithms and exponentials, especially for
the complexity analysis in the BSS model. In \S \ref{sec:interp-paths-unit} we
will see paths that are cheaper to compute but still satisfy~\ref{item:5} and~\ref{item:6}.

\begin{remark}\label{remark:computation-step-size}
  To implement Algorithm~\ref{algo:hc}, the computation of the step
  size by which~$t$ is updated can be relaxed: it is enough
  to compute a number~$\tau$ such that
  \[ M^{-1} \left(16 C \, \kappa(\bfw_t, z) g(\bfw_t, z)\right)^{-1} \leq \tau
    \leq \left(16 C \, \kappa(\bfw_t, z) g(\bfw_t, z)\right)^{-1}, \]
  for some constant~$M \geq 1$, and use it as the step size.
  Correctness is not harmed, and the number of continuation steps is now bounded
  by
  \[ 25 M C \int_0^T \kappa(\bfw_t,\zeta_t) g(\bfw_t, \zeta_t) \ud t. \]
\end{remark}

\subsection{A randomized algorithm}
\label{sec:randomized-algorithm}

We have everything we need to mimic Beltrán--Pardo algorithm in the
rigid setting: a continuation algorithm with an integral estimate for its
complexity and a recipe to sample points in the solution variety with the
appropriate distribution. This leads to Algorithm~\ref{algo:ubp} (Solve): for finding a
zero of~$\bfu \cdot F$, we first sample a random element~$(\bfv, \eta)$
in~$\cV$, with Algorithm~\ref{algo:sample} (Sample), and then perform, with Algorithm~\ref{algo:hc}, the numerical
continuation along a path in~$\cU$ from~$\bfv$ to~$\bfu$.
Based on Theorem~\ref{thm:complexity-numcont}, we can estimate the expected
number of continuation steps performed by~$\mop{Solve}(F, \bfu)$ when~$\bfu$ is uniformly distributed.

\begin{algo}[tp]
  \centering
  \begin{algorithmic}
    \Function{Solve}{$F$, $\bfu$}
      \State $(\bfv, \eta) \gets  \mop{Sample}(\cV_F)$
      \State \Return $\mop{NC}\left( \bfu, \bfv, \eta \right)$
    \EndFunction
  \end{algorithmic}
  \caption[]{An analogue of Beltrán-Pardo algorithm in the rigid setting
    \begin{description}
      \item[Input.] Homogeneous polynomial system~$F \in \cH[n]$ and~$\bfu \in \cU$
      \item[Output.] $z\in\Proj$.
      \item[Postcondition.] $z$ is an approximate zero of~$\bfu \cdot F$.
    \end{description}
  }
  \label{algo:ubp}
\end{algo}

\begin{theorem}\label{thm:expectation-K}
  If~$\bfu\in\cU$ is  uniformly distributed
  then~$\mop{Solve}(F, \bfu)$ terminates almost surely and outputs an
  approximate zero of~$\bfu\cdot F$ after~$K$ continuation steps, with
  \[ \bE \left[ K \right]\leq 100 C n \ \bE\left[\kappa(\bfv, \eta) g(\bfv,
      \eta) \right]. \]
\end{theorem}

\begin{proof}
  Let $(\bfw_t)_{0\leq t \leq T}$ be the path from~$\bfv$ to~$\bfu$ chosen in
  NC. By \ref{item:5}, $\bfv^{-1} \bfw_t$ is a function of~$\bfv^{-1} \bfu$.
  We first note that~$\bfv$ and $\bfv^{-1} \bfu$ are independent and uniformly
  distributed in~$\cU$, because the Jacobian of the diffeomorphism
  \[ (\bfu,\bfv) \in \cU \times \cU \mapsto (\bfv, \bfv^{-1} \bfu) \in \cU \times
    \cU \]
  is constant. Secondly, by hypothesis, for any~$0\leq s\leq 1$, the random
  variable $\bfv^{-1} \bfw_{Ts}$ depends only on~$\bfv^{-1}\bfu$, so it is
  independent from~$\bfv$. Therefore~$\bfw_{Ts}$, which equals~$\bfv (\bfv^{-1}
  \bfw_{Ts})$, is uniformly distributed and independent from~$\bfv^{-1} \bfu$.

  By Proposition~\ref{coro:possibility-of-continuation}, with probability~$1$, all
  the polynomial
  systems~$\bfw_t \cdot F$, for $0\leq t\leq T$, have only regular zeros. So almost
  surely, all the
  zeros of~$\bfv \cdot F$ can be continued as zeros of~$\bfw_t \cdot F$.
  Let~$\zeta_t$ be the zero of~$\bfw_t \cdot F$ obtained by continuation of the
  zero~$\eta$ of~$\bfv \cdot F$. Since~$\zeta$ is uniformly distributed among the
  zeros of~$\bfv$, it follows that~$\zeta_t$ is uniformly distributed among
  the zeros of~$\bfw_t$, because the numerical continuation, which is almost
  surely well defined, induces a bijective
  correspondance between the two finite sets of zeros.
  Therefore, for any~$0\leq s \leq 1$, $(\bfw_{Ts},\zeta_{Ts})$ is a
  $\rho_\text{std}$-distributed random variable independent from~$\bfv^{-1}
  \bfu$, and in particular, independent from~$T$.
  
  Together with Theorem~\ref{thm:complexity-numcont} (and the change of variable~$t=Ts$), this implies
  \begin{align*}
    \bE[ K ]
    & \leq \bE \left[  25 C \int_0^{1} \kappa( \bfw_{Ts}, \zeta_{Ts}) g( \bfw_{Ts}, \zeta_{Ts}) T \ud s \right] \\   
    &= 25 C \int_0^{1} \bE \left[ \kappa(\bfw_{Ts}, \zeta_{Ts}) g(\bfw_{Ts}, \zeta_{Ts}) T \right] \ud s \\
    &= 25 C \, \bE \left[ \kappa(\bfv,\eta) g(\bfv,\eta) \right] \bE[T],
  \end{align*}
  which leads to the claim, with the bound~$T \leq 4n$ given by~\ref{item:6}.
\end{proof}

\section{Average complexity for random dense polynomial systems}
\label{sec:aver-compl-dense}

Theorem~\ref{thm:expectation-K} on the average complexity of computing one zero of a random
system~$\bfu \cdot F$, with~$\bfu\in\cU$ uniformly distributed and~$F$ fixed, is applicable
to the more common question of computing one zero of a random system~$F$, under
a unitary invariance condition on the probability distribution of~$F$.

The polynomial system~$F\in\cH[n]$ fixed in~\S \ref{sec:numer-cont-algor} is now
a random variable. 
The probability distribution of~$F$ is assumed to be
\emph{unitary invariant}: that is, for any~$\bfu \in \cU$, the random systems $F$ and~$\bfu \cdot F$ are
identically distributed.

An important example of a unitary invariant distribution is \emph{Kostlan's
distribution}: that is the standard Gaussian distribution on~$\cH[n]$ endowed
with Weyl's Hermitian norm. The unitary invariance of Kostlan's distribution
follows from the invariance of Weyl's norm under the action of~$\cH$.
Concretely, a \emph{Kostlan random system} $F\in \cH[n]$ is a
tuple~$(f_1,\dotsc,f_n)$ of independent \emph{Kostlan random polynomials}, that
is
\[ f_i \eqdef \sum_{j_0+\dotsb+j_n = d_i} \left( \frac{d_i!}{j_0!\dotsb j_n!}
  \right)^{\frac12} c_{j_0,\dotsc,j_n} x_0^{j_0} \dotsb x_n^{j_n}, \]
where the coefficients~$c_{j_0,\dotsc,j_n}$ are independent standard normal
variables in~$\bC$.

Section \ref{sec:unit-invar-rand} contains the average analysis of the number of
continuation steps in the general setting of a unitary invariant distribution.
Section~\ref{sec:an-effic-comp} describes the function~$\hat\gamma_\Frob$ that
 will be used for the numerical continuation (in the role of~$g$).
Next, Section~\ref{sec:impl-deta-compl} discusses the computational model and
the construction of paths in~$\cU$.
And lastly, Section \ref{sec:gauss-rand-syst} concludes the average analysis for
Kostlan's distribution.

\subsection{Unitary invariant random systems}
\label{sec:unit-invar-rand}

Let~$\fkg : \cH[n] \to (0,\infty]$ be a function of the form (compare
to Equation~\eqref{eq:37} defining $\hat\gamma$)
\begin{equation}\label{eq:39}
  \fkg(F, z) \eqdef \kappa(F, z) \big( \fkg_1(f_1, z)^2 + \dotsb + \fkg_n(f_n, z)^2 \big)^\frac12,
\end{equation}
for some~$\fkg_i : H_{d_i} \times \Proj \to (0,\infty]$ such that for any~$1\leq i\leq n$ and any~$f \in H_{d_i}$,
\begin{enumerate}[(H1')]
\item \label{item:4} $x\in \Proj$, $\gamma(f, x) \leq
  \fkg_i(f, x)$;
\item \label{item:3} for any $x\in \Proj \mapsto \fkg_i(f,
  x)^{-1}$ is~$C'$-Lipschitz continuous ($C' \geq 4$); and
\item \label{item:7} for any~$z\in\Proj$ and~$u\in U(n+1)$,
  $\fkg_i(u\cdot f , x) = \fkg_i(f, u^{-1} x)$.
\end{enumerate}
These assumptions make it possible to use~$\fkg$ for numerical
continuation in the rigid setting.

\begin{lemma}\label{lem:lip-cond-split}
  If~\ref{item:4}--\ref{item:7} hold, then
  for any~$F\in \cH[n]$, the function
  \[ g : (\bfu, x)\in\cU \times \Proj \mapsto
  \fkg(\bfu \cdot F, x) \] satisfies conditions~\ref{item:1} and~\ref{item:2},
  with~$C=3 C'$.
\end{lemma}

\begin{proof}
  Condition~\ref{item:1} follows directly from the definitions of~$\hat\gamma$
  and~$\fkg$.
  Condition~\ref{item:2} follows exactly as the Lipschitz continuity
  of~$\hat\gamma$ (Lemma~\ref{lemma:gammas-lip}).
\end{proof}

\begin{theorem}\label{thm:complexity-uncouple}
  If the probability distribution of the system~$F = (f_1,\dotsc,f_n)\in \cH[n]$
  is unitary invariant,
  the procedure $\mop{Solve}(F, \one_\cU)$, with~$g(\bfu, x)\eqdef\fkg(\bfu\cdot F, x)$
  for the continuation,
  computes an approximate root of~$F$ using~$K$ continuation steps, with
  \[ \bE[K] \leq 1800 C' n^3 \, \bigg(\sum_{i=1}^n \bE \left[ \fkg_i(f_i,\zeta_i)^2
      \right] \bigg)^\frac12, \]
  where~$\zeta_1,\dotsc,\zeta_n\in\Proj$ are random independent uniformly
  distributed zeros of the respective random polynomials~$f_1,\dotsc,f_n$.
\end{theorem}

\begin{proof}
  Let~$K(F, \bfu)$ be the (random) number of continuation steps performed by
  $\mop{Solve}(F, \bfu)$.
  Let~$\bfu \in \cU$ be a uniformly distributed random variable independent from~$F$.
  By the unitary invariance hypothesis, $F \sim \bfu \cdot F$, so~$K(F,\one_\cU) \sim K(\bfu \cdot F,
  \one_\cU)$,
  and by the assumption~\ref{item:5} on the unitary invariance of the choice of the
  continuation path, $K(\bfu\cdot F, \one_\cU) \sim K(F, \bfu)$. In particular, $\bE\left[K(F, \one_\cU)\right] = \bE\left[K(F, \bfu)\right]$.
  By Theorem~\ref{thm:expectation-K} (with $C = 3C'$, by
  Lemma~\ref{lem:lip-cond-split}), 
  \begin{equation}\label{eq:52}
    \bE\left[K(F, \bfu)\right] \leq 300 C' n \, \bE \left[ \kappa(\bfu\cdot F, \zeta) \fkg(\bfu \cdot
      F, \zeta) \right],
  \end{equation}
  where~$\zeta$ is a uniformly distributed random zero of~$\bfu\cdot F$.
  Given the form of $\fkg$ \eqref{eq:39}, the right-hand side expands to
  \begin{equation}
    \label{eq:51}
     \kappa(\bfu\cdot F, \zeta) \fkg(\bfu \cdot
      F, \zeta) = \kappa(\bfu\cdot F, \zeta)^2 \bigg( \sum_{i=1}^n \fkg_i(f_i, u_i^{-1} \zeta)^2 \bigg)^\frac12,
  \end{equation}
  using also the unitary invariance of each~$\fkg_i$ \ref{item:7}.

  Conditionally on~$F$, $(\bfu, \zeta)$ is $\rstd$-distributed in the solution
  variety~$\cV_F$.
  By Theorem~\ref{thm:linearization},
  $u_1^{-1} \zeta, \dotsc, u_n^{-1} \zeta$ are independent and uniformly distributed
  zeros of~$f_1,\dotsc,f_n$ respectively, and $\kappa(\bfu\cdot F, \zeta)$, which depends only
  on~$L(\bfu\cdot F, \zeta)$, see~\eqref{eq:10}, is independent with them conditionally on~$F$.
  So the two factors in the right-hand side of~\eqref{eq:51} are independent
  conditionally on~$F$; therefore
  \begin{multline}\label{eq:47}
    \bE\left[ \kappa(\bfu \cdot F,\zeta) \fkg(\bfu \cdot F,\zeta) \st F \right]
    = \\
    \bE \left[ \kappa(\bfu\cdot F,\zeta)^2 \st F \right]
    \bE\left[\bigg(\sum_{i=1}^n \fkg_i(f_i,u_i^{-1} \zeta)^2 \bigg)^\frac12
      \st F \right],
  \end{multline}
  where~$\bE \left[ - | F \right]$ denotes conditional expectation.
  Proposition~\ref{prop:expectation-kappa} gives the bound $\bE \left[ \kappa(\bfu\cdot
    F,\zeta)^2 \st F \right] \leq 6n^2$,
  and then
  \begin{align}
    \bE\left[ \kappa(\bfu \cdot F,\zeta) \fkg(\bfu \cdot F,\zeta) \right] &= \bE \big[ \bE\left[ \kappa(\bfu \cdot F,\zeta) \fkg(\bfu \cdot F,\zeta) \st F \right] \big] \notag\\
    &\leq \bE \left[ 6n^2 \, \bE\left[\bigg(\sum_{i=1}^n \fkg_i(f_i,u_i^{-1} \zeta)^2 \bigg)^\frac12
      \st F \right] \right], \text{by \eqref{eq:47},}\notag\\
    &=  6n^2 \, \bE \left[\bigg(\sum_{i=1}^n \fkg_i(f_i,u_i^{-1} \zeta)^2 \bigg)^\frac12
     \right] \label{eq:48} \\
    &\leq 6n^2 \bigg(\sum_{i=1}^n \bE\left[\fkg_i(f_i,u_i^{-1} \zeta)^2 \right] \bigg)^\frac12,\notag
  \end{align}
  where the last inequality follows from Jensen's inequality.
  We note as above that~$u_i^{-1} \zeta$ is a uniformly distributed
  zero of~$f_i$, so that $ \bE\left[\fkg_i(f_i,u_i^{-1} \zeta)^2 \right] =  \bE\left[\fkg_i(f_i,\zeta_i)^2 \right]$.
  With~\eqref{eq:52}, this concludes the proof.
\end{proof}

\begin{remark}
  Under the same hypotheses, we also have the bound
  \[ \bE[K] \leq 1800 C' n^3 \, \sum_{i=1}^n \bE \left[ \fkg_i(f_i,\zeta_i)
    \right], \]
  obtained by bounding $\left(\sum_{i} \fkg_i(f_i,u_i^{-1} \zeta)^2
  \right)^\frac12 \leq \sum_i \fkg_i(f_i,u_i^{-1}\zeta)$ in~\eqref{eq:48}.
\end{remark}

\subsection{An efficiently computable variant of $\gamma$}
\label{sec:an-effic-comp}

Controlling the total complexity of Algorithm~\ref{algo:ubp}
requires a function~$g$ that is easy to compute.
The \emph{Frobenius gamma number} is introduced here with this purpose.

\subsubsection{Norms of a multilinear map}

Let~$E$ and~$F$ be Hermitian spaces and let~$h : E^k \to F$ be a multilinear map,
the \emph{operator norm} of~$h$ is defined by
\[ \tn{h} \eqdef \max \left\{ \| h(x_1,\dotsc,x_k) \| \st
    x_1,\dotsc,x_k \in \bS(E) \right\}, \]
where~$\bS(E)$ is the unit sphere of~$E$,
and the \emph{Frobenius norm} of~$h$ is defined as
\begin{equation} \|h\|_\Frob \eqdef \bigg( \sum_{1\leq i_1,\dotsc,i_k \leq n}
   \| a_{i_1,\dotsc,i_k} \|^2 \bigg)^{\frac12}, \label{eq:4}
\end{equation}
where the~$a_{i_1,\dotsc,i_k}$ are the coefficients of~$h$ in some unitary
basis~$e_1,\dotsc,e_{\dim E}$ of~$E$, that is
$a_{i_1,\dotsc,i_k} \eqdef h(e_{i_1},\dotsc,e_{i_k}) \in F$,
for~$1 \leq i_1,\dotsc,i_k \leq n$.
This definition does not depend on he choices of the unitary basis.

\begin{lemma}\label{lem:frob-op-compar}
  For any multilinear map~$h : E^k \to F$, $\tn{h} \leq \|h\|_\Frob \leq (\dim E)^{\frac
    k2} \tn{h}$.
\end{lemma}

\begin{proof}
  This is better seen if we consider~$h$ as a map~$E \to ( E \to \dotsb (E\to
  F))$. The claim follows from an induction on~$k$ and the usual
  comparison between the Frobenius and the operator norm.
\end{proof}



The following lemma relates the Weyl norm of a homogeneous polynomial with the Frobenius
norm of an appropriate higher derivative.
Recall that the Weyl norm of a homogeneous polynomial of degree~$k$ is defined
by
\begin{equation}\label{eq:41}
 \bigg\|\sum_{j_0+\dotsb+j_n = k} c_{j_0,\dotsc,j_{n}} x_0^{j_0}\dotsm
    x_n^{j_n} \bigg\|^2_W \eqdef \sum_{j_0+\dotsb+j_n = k}\frac{j_0! \dotsb j_{n}!}{k!}
  | c_{j_0,\dotsc,j_{n}}|^2.
\end{equation}

\begin{lemma}\label{lem:weyl-frob-norms}
  For any homogeneous polynomial~$p(x_0,\dotsc,x_n)$ of degree~$k$,
  \begin{equation*}
    \|p\|_W = \left\| \tfrac{1}{k!}\ud_0^k p \right\|_\Frob.
  \end{equation*}
\end{lemma}

\begin{proof}
  For any $0\leq i_1,\dotsc,i_k \leq n$,
  \begin{align*}
    \frac{1}{k!} \ud_0^k p(e_{i_1},\dotsc,e_{i_k})
    &= \left. \frac{1}{k!} \frac{\partial^k p}{\partial x_{i_1} \dotsm \partial x_{i_k}}  \right|_{x=0} \\
    &= \frac{j_0! \dotsb j_{n}!}{k!} c_{j_0,\dotsc,j_{n}},
  \end{align*} 
  where~$j_m$ is the number of indices~$i_*$ that are equal to~$m$
  and where~$c_{j_0,\dotsc,j_{n}}$ is the coefficient of~$x_0^{j_0}\dotsb
  x_n^{j_n}$ in~$p$.
  There are exactly~$\frac{k!}{j_0! \dotsb j_{n}!}$ $k$-uples $i_*$ that lead to
  a given $(n+1)$-uple $j_*$. Therefore,
  \begin{equation*}
    \left\| \tfrac{1}{k!} \ud_0^k p(e_{i_1},\dotsc,e_{i_k}) \right\|^2_\Frob = \sum_{j_0+\dotsb+j_n = k} \frac{j_0! \dotsb j_{n}!}{k!} | c_{j_0,\dotsc,j_{n}}|^2,
  \end{equation*}
  and this is exactly~$\|p\|^2_W.$
\end{proof}

\subsubsection{The $\gamma$ number with Frobenius norms}
\label{sec:gamma-number-with}

For a homogeneous polynomial~$f : \bC^{n+1} \to \bC$, we define
\begin{equation}\label{eq:36}
  \gamma_\Frob(f,z) \eqdef
  \begin{cases}
    \sup_{k\geq 2} \left(\tfrac{1}{k!} \left\| \ud_zf \right\|^{-1} \left\| \ud_z^kf
    \right\|_\Frob\right)^{\frac{1}{k-1}} & \text{if $\ud_zf$ is
    nonzero,} \\
  \infty & \text{otherwise},
  \end{cases}
\end{equation}
and for a homogeneous polynomial system~$F = (f_1,\dotsc,f_r) \in \cH[r]$, we define
\begin{equation*}
\hat\gamma_\Frob(F,z) \eqdef \kappa( F, z) \left(\gamma_\Frob(f_1, z)^2 + \dotsb + \gamma_\Frob(f_r, z)^2 \right)^{\frac12}.
\end{equation*}
Compare with the definitions of~$\gamma$
and~$\hat \gamma$, \S \ref{sec:splitgamma}.

In order to use Algorithm~\ref{algo:hc}
with $\hat\gamma_\Frob$ as~$g$, we must check~\ref{item:1} and~\ref{item:2}.
By Lemma~\ref{lem:lip-cond-split}, it is enough to check~\ref{item:4},
\ref{item:3} and~\ref{item:7} for~$\gamma_\Frob$ of a single polynomial.
The condition~\ref{item:4}, that is~$\gamma \leq \gamma_\Frob$, is clear by
Lemma~\ref{lem:frob-op-compar}.
The condition~\ref{item:7}, that is~$\gamma_\Frob(u\cdot f, z) = \gamma_\Frob(f,
u^{-1} z)$ follows from the unitary invariance of the Frobenius norm. The following lemma shows \ref{item:3}.

\begin{lemma}\label{lem:gamma-frob-lip}
  For any polynomial~$f\in H_d$, the function~$x \in
  \Proj\mapsto 1/\gamma_\Frob(f,x)$ is $5$-Lipschitz continuous.
\end{lemma}

\begin{proof}
  Let~$(x_t)_{0\leq t\leq 1}$ be a differentiable path in~$\bS(\bC^{n+1})$. Let~$a_t
  \eqdef \|\ud_{x_t} f\|^{-1}$ and~$(B_k)_t \eqdef \frac{1}{k!} a_t d^{k}_{x_t}
  f$, so that $(B_k)_t$ is a multilinear map~$(\bC^{n+1})^k\to \bC$.
  From now on, we drop the index~$t$ and denote the derivative with respect
  to~$t$ with a dot or with~$\frac{\ud}{\ud t}$. Writing~$a$ as~$\langle \ud_x
  f, \ud_xf \rangle^{-\frac12}$, we compute that
  \begin{equation}\label{eq:32}
    \dot a = - a\, \Re \langle a\, \ud_x^2 f(\dot x), a\, \ud_x f \rangle,
  \end{equation}
  where~$\Re$ denotes the real part, and the Cauchy--Schwartz inequality
  implies that
  \begin{equation}\label{eq:33}
    |\dot a| \leq a \, \|a\, \ud_x^2 f(\dot x)\|,
  \end{equation}
  noting that $\|a\, \ud_x f\| = 1$.
  By definition of~$a$ and~$\gamma(f, x)$,
  \begin{align*}
    \| a\, \ud_x^2 f(\dot x) \|
    &\leq \| \ud_x f \|^{-1} \tn{\ud_x^2 f} \|\dot x\| \\
    &\leq \| \ud_x f \|^{-1} \|\ud_x^2 f\|_\Frob \|\dot x\|, &&\text{by Lemma~\ref{lem:frob-op-compar},}\\
    &\leq 2 \gamma_\Frob(f, x) \|\dot x\|, &&\text{by definition of~$\gamma_\Frob(f,x)$,}
  \end{align*}
  which implies, in combination with~\eqref{eq:33}, that
  \begin{equation}\label{eq:40}
    | \dot a | \leq 2 a \gamma_\Frob(f,x) \|\dot x\|.
  \end{equation}
  We note that~$\tdert \ud_x^k f = \ud_x^{k+1} f(\dot x)$,
  therefore,
  \begin{equation}\label{eq:34}
    \dot B_k = \dot a a^{-1} B_{k} + (k+1) B_{k+1}(\dot x),
  \end{equation}
  and it follows, since for any~$l$, $\|B_l\|_\Frob \leq
  \gamma_\Frob(f,x)^{l-1}$, that
  \begin{align*}
    \abs{\tdert \| B_k \|_\Frob}
    &\leq \|\dot B_k\|_\Frob, \quad \text{as $\|-\|_\Frob$ is 1-Lipschitz continuous,}\\
    &\leq 2 \gamma_\Frob(f,x) \|B_k\|_\Frob \|\dot x\| + (k+1)\|B_{k+1}\|_\Frob \|\dot x\|, \\
   &&  \text{\llap{ by~\eqref{eq:34} and~\eqref{eq:40},}} \\
    &\leq \diffblock{\left( 2 \gamma_\Frob(f,x)^k + (k+1) \gamma_\Frob(f,x)^k \right) \|\dot x\|} \\
    &\leq \diffblock{(k+3) \gamma_\Frob(f,x)^{k} \|\dot x\|}.
  \end{align*}
  It follows that
  \begin{equation}\label{eq:45}
  \abs{\tdert \| B_k \|_\Frob^{\frac{1}{k-1}}} = \frac{1}{k-1} \|B_k
    \|_\Frob^{\frac{1}{k-1}-1} \abs{\tdert \| B_k \|_\Frob} \leq \frac{k+3}{k-1}
    \frac{\gamma_\Frob^{k+1}}{\|B_k\|_\Frob} \|\dot x\|.
  \end{equation}
  By definition, $\gamma_\Frob$ is the supremum of all~$\|B_k\|_\Frob^{1/(k-1)}$
  ($k\geq 2$), finitely many of which are nonzero, so at a given time~$t$,
  there is some~$k$ such that
  \begin{equation}\label{eq:46}
    \gamma_\Frob(f,x) = \|B_k\|_\Frob^{1/(k-1)}
  \end{equation}
  in a (one
  sided) neighborhood of $t$. Therefore $\dot \gamma_\Frob=\tdert \| B_k
  \|_\Frob^{1/(k-1)}$ and the computation above shows that
  \begin{align*}\label{eq:35}
    \abs{\dot \gamma_\Frob} &\leq  \frac{k+3}{k-1} \frac{\gamma_\Frob^{k+1}}{\|B_k\|_\Frob} \|\dot x\|, && \text{by \eqref{eq:45},} \\
                            &= \frac{k+3}{k-1} \frac{\gamma_\Frob^{k+1}}{\gamma_\Frob^{k-1}} \|\dot x\|, && \text{by \eqref{eq:46}},\\
                            &= \frac{k+3}{k-1} \gamma_\Frob^2 \|\dot x\| \\
                            &\leq 5 \gamma^2_\Frob \|\dot x\|, && \text{because $k\geq 2$.}
  \end{align*}
  As a function on~$\Proj$, $\ud_x \gamma_\Frob(f, -) (\dot x) = \dot\gamma_\Frob$,
  so that
  \[ \abs{\ud_x \gamma_\Frob(f, -) (\dot x) } \leq 5 \gamma_\Frob(f, x)^2 \|\dot x\|.\]
  Since the inequality holds for any differentiable path, covering all possible
  values of~$\dot x$, it follows that $\|\ud_x \gamma_\Frob(f, -)\| \leq 5 \gamma_\Frob(f,
  x)^2$.
  Consequently, $\|\ud_x \gamma_\Frob(f, -)^{-1}\| \leq 5$ and the claim
  follows.
\end{proof}

\subsection{Implementation details, complexity}
\label{sec:impl-deta-compl}

\subsubsection{Computational model}
\label{sec:computational-model}

We use the Blum--Shub--Smale model \parencite{BlumShubSmale_1989} extended with
a ``6th type of node'', as did \textcite{ShubSmale_1996}. Unlike Shub and Smale,
we will apply it to univariate (or rather homogeneous bivariate) polynomials
only. A node of this type has the following behavior. If it is given as input a
homogeneous polynomial~$f\in\bC[x,y]$ and an approximate zero~$z\in\bP^1$
of~$f$, with associated zero~$\zeta$, it outputs~$\zeta$. In any other case, it
fails. There is no need to specify how it fails because we will make sure that
this will not happen.

From the practical point of view, given a point~$z$ which approximates a
zero~$\zeta$ of a homogeneous polynomial~$f\in\bC[x,y]$, one can refine the
approximation to obtain~$d_\bP(z,\zeta) \leq \epsilon$ in~$\log_2 \log_2
\frac{\pi}{\epsilon}$ Newton's iterations. For most practical purpose, this
looks like infinite precision. In that sense, the 6th type of node does not add
much power.

Do we really need a 6th type of node? The continuation method proposed here uses
a start system defined in terms of the zeros of some homogeneous bivariate
polynomials. Naturally, the algorithm would also work with approximate zeros
only. However, if we do it this way, then the distribution of the start system
is not easily described, it is only close to a nice distribution. I showed
\parencite{Lairez_2017} how to deal with the complexity analysis in an analogue
situation but it is too technical an argument for the little value it adds.

Interval arithmetic gives another way to remove the need for this extra type of
node: wherever an exact zero is expected, we use bounding boxes instead and
perform the subsequent operations with interval arithmetic. If the precision
happens to be insufficient, we refine the bounding boxes with Newton's iteration
and start over the computation. The convergence of Newton's iteration is so fast
that even with naive estimations of the numerical stability, the number of start
over will be moderate. However, this is no less technical to formalize.

For convenience, we will also assume the ability to compute fractional powers of
a positive real number at unit cost. This will allow us to compute 
Hermitian norms and the numbers~$\gamma_\Frob$ and~$\hat\gamma_\Frob$ exactly.

\subsubsection{Continuation paths in the unitary group}
\label{sec:interp-paths-unit}

While geodesics in~$\cU$ are a
natural choice for continuation paths, see~\S \ref{sec:yet-anoth-cont}, they are not easy to compute
in the BSS model. We can describe  more elementary continuation paths using Householder's
reflections. For any 1-dimensional subspace~$l\subset\bC^{n+1}$ and any~$\theta\in \bR$, let~$R(l,\theta) \in U(n+1)$ be
the unique map such that~$R(l,\theta)|_l = e^{i\theta} \Id_l$ and
$R(l,\theta)|_{l^\perp} = \mop{id}_{l^\perp}$. Note that for the angle~$\pi$,
$R(l,\pi)$ is a reflection.
The computation of a matrix multiplication~$A R(l, \theta)$, for any~$A\in
U(n+1)$, requires only~$O(n^2)$ operations because~$R(l, \theta) -
\mop{id}$ has rank~$1$ and can be written~$v v^*$ for some~$v\in
\bC^{(n+1)\times 1}$.

Given a unitary matrix~$v\in U(n+1)$, the procedure of
\textcite{Householder_1958} (with the necessary changes in the complex case)
decomposes~$v$ as
$v = e^{i\alpha} R(l_1, \pi) \dotsb R(l_n,\pi)$,
for some~$\alpha \in [-\pi,\pi]$, with~$O(n^3)$ operations.
One can define the path
\[ w_t \eqdef e^{i \tfrac{t}{\tau} \alpha} R(l_1, \tfrac{t}{\tau} \pi) \dotsb
  R(l_n, \tfrac{t}{\tau} \pi), \]
where~$\tau^2 \eqdef \frac12 \left( \alpha^2 + n \pi^2 \right) \leq
\frac{n+1}{2} \pi^2$. Given~$\bfu,\bfv \in \cU$, we define a $1$-Lipshitz continuous
path~$(\bfw_t)_{t\geq 0}$ from~$\one_\cU$ to~$\bfv^{-1}\bfu$, component-wise with the method
above. It reaches~$\bfv^{-1}\bfu$ at~$t =\sqrt{\frac{n(n+1)}{2}} \pi < 4n$.
To compute~$\bfw_t$ on a BSS machine, we can replace the trigonometric
functions with any other functions parametrizing the circle.
This construction satisfies the three conditions~\ref{item:9}--\ref{item:6} from~\S\ref{sec:yet-anoth-cont}.

For a given~$t$, the cost of computing~$\bfw_t$ is dominated by the
the cost of multiplying by~$R(l, \theta)$ matrices: there are $n$ such
multiplications for each of the~$n$ components of~$\bfw_t$, that is~$O(n^4)$
operations.

\subsubsection{Computation of $\hat\gamma_\Frob$, cost of a continuation step}
\label{sec:comp-hatg-cost}

The reason for introducing~$\hat \gamma_\Frob$ is that we can compute it with low
complexity. By contrast, computing~$\hat \gamma$ is NP-hard because it involves the
computation of the spectral norm of symmetric multilinear maps, and there is no
polynomial-time approximation scheme, unless $\text{P}=\text{NP}$ \parencite[Theorem~10.2]{HillarLim_2013}.
Beware though, a too naive algorithm for computing~$\gamma_\Frob$
requires~$\Omega(N^2)$ operations, where $N$ is the input size (see \S\ref{sec:notat-basic-defin}).
Given a polynomial~$f \in H_d$ and a point~$z\in\bC^{n+1}$, the
evaluation~$f(z)$ and the vector~$\ud_z f$ can be computed with~$O(\dim H_d)$
operations \parencites[Lemma~16.32]{BurgisserCucker_2013}{BaurStrassen_1983}.

\begin{proposition}\label{prop:complexity-gammafrob}
  Given a homogeneous polynomial~$f \in H_d$ and~$z\in \Proj$, one can compute
  $\gamma_\Frob(f, z)$ with~$O(n d^2 \dim H_d)$ operations, as~$\dim H_d
  \to \infty$.
\end{proposition}

\begin{proof}
  We can compute $\|\ud_z f\|^{-1}$ in $O(\dim H_d)$ operations,
  so the main task is to compute~$\left\| \frac{1}{k!} \ud_z^kf \right\|_\Frob$ for all~$2 \leq
  k \leq d$, see~\eqref{eq:36}. Let~$g$ be the shifted polynomial~$g :x \in \bC^{n+1}\mapsto
  f(z+x)$, so that
  $\ud_0^k g = \ud_z^k f$.
  Let~$g_k$ denote the homogeneous component of degree~$k$ of~$g$.
  According to Lemma~\ref{lem:weyl-frob-norms},
  $\| \frac{1}{k!} \ud_0^k g\|_\Frob = \|g_k\|_W$.

  Let~$S$ denote the \emph{size} of~$g$, that is
  the number of coefficients in a dense nonhomogeneous polynomial of degree~$d$ in $n+1$
  variables. He have~$S = \binom{n+d+1}{d} \leq d \dim H_d$.
  In view of~\eqref{eq:41}, the computation of
  $\|g_k\|_W^2$ reduces to the computation of the coefficients
  of~$g_k$ in the monomial basis and the multinomial
  coefficients $\frac{i_0!\dotsb i_n!}{(i_0+\dotsb+i_n)!}$.
  We can compute them all with $O(S)$ operations thanks to the
  recurrence relation
  \[ \frac{i_0!\dotsb i_{n-1}! (i_n+1)!}{(i_0+\dotsb+i_n+1)!} =
    \frac{i_n+1}{i_0+\dotsb+i_n+1}\frac{i_0!\dotsb i_n!}{(i_0+\dotsb+i_n)!}. \]
  Therefore, we can compute~$\gamma_\Frob(f,z)$ in
  $O(S)$ operations given the coefficients of~$g$ in the monomial basis.
  
  To compute~$g$, we shift the variables one after the other.
  To compute the first shift~$f(x_0+z_0, x_1,\dotsc,x_n)$ (subsequent shifts are
  identical), we write
  \[ f = \sum_{i_1+\dotsb+i_n \leq d} p_{i_1,\dotsc,i_n}(x_0) x_1^{i_1}\dotsb
    x_n^{i_n}, \]
  where the~$p_{i_1,\dotsc,i_n}(x_0)$ are polynomials of degree at most~$d$.
  There are at most~$\dim H_d$ of them (the number of monomials of degree at most~$d$
  in~$n$ variables). Moreover, computing them requires only $O(S)$ copy
  operations and no arithmetic
  operation: their coefficients in the monomial basis are directly read from the coefficients
  of~$f$ in the monomial basis. One can compute $p_{i_1,\dotsc,i_n}(x_0+z_0)$ with~$O(d^2)$
  operations with a naive algorithm. Note that we can do this with only~$d^{1+o(1)}$
  operations using fast evaluation and interpolation algorithms
  \parencite{BostanChyzakGiustiEtAl_2017,GathenGerhard_1999}.
  All together, this is~$O(d^2 \dim H_d)$ operations.
  We recover
  $f(x_0+z_0, x_1,\dotsc,x_n)$ in $O(S)$ operations from the
  $p_{i_1,\dotsc,i_n}(x_0 + z_0)$ (and no arithmetic operation). We repeat this shift
  operation for each one of the $n+1$ variables and this gives the claim.
\end{proof}

\begin{corollary}\label{coro:kappa-gamma-complexity}
  Given~$F\in\cH[n]$, $\bfu\in\cU$ and~$z\in\Proj$, one can compute~$\kappa(\bfu
  \cdot F, z)$
  and~$\hat\gamma_\Frob(\bfu\cdot F, z)$ within a factor~$2$ in~$O(nD^2 N)$ operations as~$N\to \infty$.
\end{corollary}

\begin{proof}
  The evaluation of~$\bfu \cdot F$ at a point~$z\in\bC^{n+1}$ can be computed as
  \[ \left( f_1(u_1^* z),\dotsc,f_n(u_n^* z) \right) \]
  with~$O(n^3)$ operations to compute the vectors~$u_i^* z$ and~$\sum_i O(\dim
  H_{d_i}) = O(N)$ additional operations to evaluate the polynomials~$f_i$.
  Therefore, the matrix~$\ud_z(\bfu\cdot F)$ be computed with~$O(N + n^3)$
  operations too \parencite{BaurStrassen_1983}.
  The computation of~$L(\bfu \cdot F, z)$ requires~$O(n^2)$ additional
  operations, following~\eqref{eq:25}. 
  Then, we can compute $\kappa$, that is the inverse of the least
  singular value of~$L(\bfu, z)$, within a factor of~$2$
  in $O(n^3)$ operations, using a
  tridiagonalization with Householder's reflections and a result by
  \textcite{Kahan_1966}.
  To compute~$\hat \gamma_\Frob$, it only remains to compute
  the~$\gamma_\Frob(u_i\cdot f, z) = \gamma_\Frob(f, u_i^* z)$, for $1\leq i\leq
  n$, which requires
  all together~$O(nD^2N + n^3)$ operations, by
  Proposition~\ref{prop:complexity-gammafrob}.
  The term~$n^3$ is~$O(nN)$, so this
  gives the claim.
\end{proof}

Since~$\kappa$ is defined as the operator norm of some
matrix, the exact computation is difficult in the BSS model.
One could use the Frobenius norm instead but computing~$\kappa$ within a
factor~2 is satisfactory, see Remark~\ref{remark:computation-step-size}.

\begin{corollary}\label{coro:complexity-splitgamma}
  In Algorithm~\ref{algo:hc} with $g = \hat \gamma_\Frob$,
  one continuation step can be performed in $O(n^2D^2N)$ operations.
\end{corollary}

\begin{proof}
  A step boils down to: one evaluation of~$\hat \gamma_\Frob$ and~$\kappa$, that
  is~$O(nD^2N)$ operations by Corollary~\ref{coro:kappa-gamma-complexity};
  one evaluation of the continuation path, that is~$O(n^4) = O(n^2N^2)$ operations, see~\S \ref{sec:interp-paths-unit};
  and one Newton's iteration, that is~$O(n^3 + N) = O(nN)$ operations \parencite[Proposition~16.32]{BurgisserCucker_2013}.
\end{proof}

\subsection{Gaussian random systems}
\label{sec:gauss-rand-syst}

We conclude with the study of the average complexity of~$\mop{Solve}(F,
\one_\cU)$ when~$F$ is a \emph{Kostlan random system}, that is a standard normal
variable in~$\cH[n]$ endowed with Weyl's norm.
In view of Theorem~\ref{thm:complexity-uncouple}, it only remains to study the
average value of~$\gamma_\Frob(f, \zeta)$ when~$f$ is a Kostlan random
polynomial and~$\zeta$ a uniformly distributed zero of it.
To this purpose, the main tool is a corollary of a result by~\textcite{BeltranPardo_2011}.

\begin{proposition}\label{prop:bp-1eq}
  Let~$f \in H_d$ be a Kostlan random polynomial, let~$\zeta$ be a random
  uniformly distributed point in $\left\{ z\in\Proj\st f(z)=0 \right\}$ and
  let~$\bar\zeta\in\bC^{n+1}$ be a random uniformly distributed vector such
  that~$\|\bar \zeta\| =1$ and~$[\bar\zeta]=\zeta$.
  Then
  \begin{enumerate}[(i)]
  \item\label{item:10} $\frac{1}{\sqrt{d}} \ud_\zeta f$ is a standard normal variable in
    $(\bC^{n+1})^*$; and
  \item\label{item:11} given $\zeta$, the orthogonal projection of~$f$ on~$\left\{ g\in H_d \st
      g(\zeta) = 0 \text{ and } \ud_\zeta g = 0\right\}$ is a
    standard normal variable
    and is independent from~$\ud_\zeta f$.
  \end{enumerate}
\end{proposition}

\begin{proof}
  Let~$\lambda_2, \dotsc,\lambda_n \in (\bC^{n+1})^*$ be independent Gaussian
  linear forms,
  so that~$F \eqdef (f, \lambda_2, \dotsc, \lambda_n)$ is a Kostlan random
  system of degree~$(d,1,\dotsc,1)$.
  Let~$\eta \in \Proj$ be a uniformly distributed random zero of~$F$.

  The zero set~$L$ of~$\lambda_2,\dotsc,\lambda_n$
  is a uniformly distributed random line independent from~$f$
  and~$\eta$ is uniformly distributed in~$V(f) \cap L$.
  By Corollary~\ref{coro:uniform-sampling}, $\eta$ is uniformly distributed
  in~$V(f)$, so we may assume that~$\zeta = \eta$.

  Let~$G$ the orthogonal projection of~$F$ on the subspace
  \[ R_\zeta \eqdef \left\{ G\in \cH \st G(\zeta) = 0 \text{ and } \ud_\zeta G =
      0\right\}. \]
  \Textcite[Theorem~7]{BeltranPardo_2011}, and
  \textcite[Prop.~17.21]{BurgisserCucker_2013} for the Gaussian case, proved that: the matrix
  $\mop{diag}(d^{-\frac12},1,\dotsc,1) \ud_{\bar \zeta} F$ is a standard normal
  variable in~$\bC^{n\times (n+1)}$; and conditionally on~$\zeta$, $G$ is a
  standard normal variable in~$R_\zeta$ and is independent from~$\ud_\zeta F$.
  The claim follows by considering the first row of~$\ud_{\bar\zeta} F$, that
  is~$\ud_{\bar\zeta} f$, and the first coordinate of~$G$.
\end{proof}

The following three lemmas deal with the average
analysis of~$\gamma_\Frob$.

\begin{lemma}\label{lem:bound-norm-derivative}
  Let~$f \in H_d$ and~$\zeta = [1:0:\dotsb:0]$.
  We write
  \[ f = \sum_{i=0}^d x_0^{d-i} g_i(x_1,\dotsc,x_n), \]
  for some uniquely determined homogeneous polynomials~$g_0,\dotsc,g_d$ of
  degrees~$0,\dotsc,d$ respectively.
  For any~$k\geq2$,
  \[     \left\| \tfrac{1}{k!} \ud^k_\zeta f \right\|_\Frob^2 = \binom{d}{k} \sum_{l = 0}^k
    \binom{d-l}{k-l} \left\|x_0^{d-l} g_l \right\|^2_W. \]
\end{lemma}

\begin{proof}
  Let~$\tilde f \eqdef
  f(x_0+1,x_1,\dotsc,x_n)$ and let~$\tilde f_{(k)}$ be the homogeneous component
  of degree~$k$ of~$\tilde f$. We compute that
  \[ \tilde f_{(k)} = \sum_{l=0}^d \left[ (x_0+1)^{d-l} g_l \right]_{(k)}= \sum_{l = 0}^k \binom{d-l}{k-l} x_0^{k-l}
    g_l. \]
  The terms of the sum are orthogonal for Weyl's inner product, and moreover $
  \big\| \tfrac{1}{k!} \ud^k_\zeta f \big\|_\Frob = \big\| \tilde f_{(k)}
  \big\|_W$ by Lemma~\ref{lem:weyl-frob-norms}, therefore
  \begin{equation*}\label{eq:11}
    \left\| \tfrac{1}{k!} \ud^k_\zeta f \right\|_\Frob^2 = \sum_{l = 0}^k
    \binom{d-l}{k-l}^2 \left\|x_0^{k-l} g_l \right\|^2_W.
  \end{equation*}
  Looking closely at the definition of Weyl's inner product
  reveals that
  \[ \left\| x_0^{k-l} g_l \right\|_W^2 =
    \frac{\binom{d}{d-l}}{\binom{k}{k-l}} \left\| x_0^{d-l} g_l \right\|_W^2
    = \frac{\binom{d}{k}}{\binom{d-l}{k-l}} \left\| x_0^{d-l} g_l \right\|_W^2, \]
  and the claim follows.
\end{proof}

\begin{lemma}\label{lem:expect-norm-derivative}
  Let~$f \in H_d$ be a Kostlan random polynomial and~$\zeta \in \Proj$ be a
  uniformly distributed zero of~$f$. For any~$k \geq 2$,
  \[ \bE \left[ \left\| \ud_\zeta f \right\|^{-2} \left\| \tfrac{1}{k!} \ud^k_\zeta f_i \right\|_\Frob^{2}
    \right] \leq \frac{1}{nd} \binom{d}{k} \binom{d+n}{k} \leq \left(
      \tfrac14 d^2(d+n)  \right)^{k-1} . \]
\end{lemma}

\begin{proof}
  We can choose (random) coordinates such that~$\zeta = [1:0:\dotsb:0]$.
  Let~$g_0,\dotsc,g_d$ be the polynomials defined in
  Lemma~\ref{lem:bound-norm-derivative}. To begin with, we describe the
  distribution of the random polynomials~$g_0,\dotsc,g_d$. The polynomial $g_0$
  (of degree $0$) is simply zero because~$\zeta$ is a zero of~$f$. The second
  one~$g_1$ has degree~$1$, and as a linear form, it is equal to~$\ud_\zeta f$.
  According to Proposition~\ref{prop:bp-1eq}\ref{item:10}, it is a
  standard normal variable after multiplication
  by~$d^{-\frac12}$. In particular~$d \|g_1\|_W^{-2}$ is the inverse of a
  $\chi^2$-distributed variable with~$2n+2$ degrees of freedom, therefore
  \begin{equation*}
    \bE \left[ \| g_1 \|_W^{-2} \right] = \frac{1}{2nd}.
  \end{equation*}
  For~$l \geq 2$, the polynomial $x_0^{d-l} g_l$
  is the orthogonal projection of~$f$ on the subspace of all 
  polynomials of the form~$x_0^{d-l}
  p(x_1,\dotsc,x_n)$, a which is a subspace of
  \[ \left\{ g\in H_d\st g(\zeta) = 0 \text{ and
    } \ud_\zeta g = 0 \right\} \subset H_d. \]
  Thus, by Proposition~\ref{prop:bp-1eq}\ref{item:11}, $x_0^{d-l} g_l$ is a
  standard normal variable in the appropriate subspace and is independent from~$g_1$.
  In particular, for any~$l \geq 2$,
  \[ \bE\left[ \left\| x_0^{d-l} g_l \right\|_W^2 \right] = \dim_\bR \left\{ x_0^{d-l}
      p(x_1,\dotsc,x_n)\in H_d \right\} = 2\binom{n-1+l}{l}. \]
  Note also that~$\left\| g_1 \right\|_W^{-2} \left\|x_0^{d-l} g_1 \right\|_W^2 = \frac1d$.
  It follows, by Lemma~\ref{lem:bound-norm-derivative}, that
  \begin{align*}
    \bE &\left[ \left\| \ud_\zeta f \right\|^{-2}  \left\| \tfrac{1}{k!} \ud^k_\zeta f \right\|_\Frob^2 \right]
    = \sum_{l=0}^k \bE \left[ \binom{d}{k} \sum_{l=0}^k  \binom{d-l}{k-l} \left\| g_1 \right\|_W^{-2} \left\|x_0^{d-l} g_l \right\|^2_W  \right] \\
    &\leq \binom{d}{k} \left(  \binom{d-1}{k-1} \frac{1}{d} + \sum_{l=2}^k \binom{d-l}{k-l}\bE \left[ \|g_1\|_W^{-2} \right] \bE \left[ \| x_0^{d-l} g_l \|_W^2 \right]  \right) \\
    &\leq \binom{d}{k} \left( \binom{d-1}{k-1} \frac1d  + \frac{1}{nd} \sum_{l = 2}^k \binom{d-l}{k-l} \binom{n-1+l}{l} \right) \\
    &\leq \frac{1}{nd} \binom{d}{k} \sum_{l = 0}^k \binom{d-l}{k-l} \binom{n-1+l}{l} \\
    &= \frac{1}{nd} \binom{d}{k} \binom{d+n}{k}.
  \end{align*}
  To check the binomial identity~$\sum_{l = 0}^k \binom{d-l}{k-l}
  \binom{n-1+l}{l} = \binom{d+n}{k}$, we remark that $\binom{d-l}{k-l}$ counts the
  number of monomials of degree~$k-l$ in~$d-k+1$ variables while $\binom{n+l}{l}$
  counts the number of monomials of degree~$l$ in~$n+1$ variables.
  Therefore, the sum over~$l$ counts the number of monomials of degree~$k$
  in~$(d-k+1)+(n+1)$ variables, that is~$\binom{d+n}{k}$.

  Concerning the second inequality, the maximum value
  of $\left( \frac{1}{nd} \binom{d}{k} \binom{d+n}{k}
  \right)^{\frac{1}{k-1}}$, with~$k\geq 2$,
  is reached for~$k=2$. That is, for any~$k\geq 2$,
  \begin{align*}
  \frac{1}{nd} \binom{d}{k} \binom{d+n}{k} &\leq \left( \frac{1}{nd} \binom{d}{2} \binom{d+n}{2} \right)^{k-1} \\
                                           &\leq \left( \tfrac14 (d-1)  (d+n) \left( \frac{d-1}{n} +1 \right)  \right)^\frac{1}{k-1},
  \end{align*}
  which leads to the claim.
\end{proof}

\begin{lemma}\label{lem:expect-gamma-frob}
  With the same notations as Lemma~\ref{lem:expect-norm-derivative},
  \[ \bE \left[ \gamma_\Frob(f,\zeta)^2 \right] \leq \tfrac14 d^3 (d+n). \] 
\end{lemma}

\begin{proof}
  We bound the supremum in the definition of~$\gamma_\Frob$ by a sum:
  \begin{align*}
    \bE  \left[ \gamma_\Frob(f,\zeta)^2 \right]
    &\leq \sum_{k = 2}^d \bE \left[ \left( \left\| \ud_\zeta f \right\|^{-1} \left\| \tfrac{1}{k!} \ud^k_\zeta f_i \right\|_\Frob \right)^{\frac{2}{k-1}} \right] \\
    &\leq  \sum_{k = 2}^d \bE \left[  \left\| \ud_\zeta f \right\|^{-2} \left\| \tfrac{1}{k!} \ud^k_\zeta f_i \right\|_\Frob^{2}  \right]^{\frac{1}{k-1}}, \text{ by Jensen's inequality,} \\
    &\leq \sum_{k=2}^d  \tfrac14 d^2 (d+n), \text{ by Lemma~\ref{lem:expect-norm-derivative},}\\
    &\leq \tfrac14 d^3 (d+n),
  \end{align*}
  and this is the claim.
\end{proof}

\begin{proposition}\label{prop:complexity-sampling-gaussian}
  If~$F \in \cH[n]$ is a Kostlan random system, then~$\mop{Sample}(F)$
  (Algorithm~\ref{algo:sample})
  outputs a $\rstd$-distributed~$(\bfu, \zeta)\in \cV_F$ with~$O(n^3)$ samplings
  of the normal distribution on~$\bR$ and~$O(n^4+nD^4)$ operations on average.
\end{proposition}

\begin{proof}
  By Proposition~\ref{prop:complexity-sampling},
  $\mop{Sample}(F)$ requires~$O(n^3)$ samplings, $O(n^4)$ operations and~$n$
  times root-finding of bivariate homogeneous polynomials of degree at most~$D$.
  These polynomials are the restrictions of the polynomials~$f_1,\dotsc,f_n$ on
  independent uniformly distributed random projective line~$L_1,\dotsc,L_n$ respectively.
  Choosing random coordinates so that~$L_i$ is spanned by~$[1:0:\dotsb]$
  and~$[0:1:0:\dotsb]$, the restriction of~$f_i$ to~$L_i$ is just
  $f_i(x_0,x_1,0,\dotsc,0)$. So the restriction map~$f\in H_{d_i} \to
  f_{|L_i}$ is an orthogonal projector (or equivalently, the adjoint map~$f \in
  \bC[x_0,x_1]_d \to H_d$ is an isometric embedding, which follows from the
  definition~\eqref{eq:41} of Weyl's norm).
  In particular, $f_i|_{L_i}$ is a Kostlan random polynomial in~$\bC[x_0,x_1]_d$,
  and the algorithm of \textcite{BeltranPardo_2011} computes an approximate root of
  it with $O(D^4)$ operations on average. The 6th type of node recover the exact root.
\end{proof}

\begin{theorem}[Main result]\label{thm:main-result}
  If~$F\in \cH[n]$ is a Kostlan random system, then $\mop{Solve}(F, \one_\cU)$
  (with $g(\bfu, x) \eqdef \hat\gamma_\Frob(\bfu\cdot F, x)$ for the continuation)
  outputs an approximate zero of~$F$ almost surely
  with $O(n^4 D^2)$ continuation steps on average and $O(n^6 D^4 N)$
  operations on average, when~$N\to \infty$.
  When~$\min(n,D)\to \infty$, this is~$N^{1+o(1)}$.
\end{theorem}

\begin{proof}
  We first note that the conditions~\ref{item:4}--\ref{item:7} are statisfied
  for~$\hat\gamma_\Frob$, see~\S \ref{sec:gamma-number-with},
  and that the probability distribution of~$F$ is unitary invariant
  (because~$\cU$ acts isometrically on~$\cH[n]$), therefore
  Theorem~\ref{thm:complexity-uncouple} applies.
  
  For~$1\leq i\leq n$, let~$\zeta_i \in \Proj$ be a uniformly distributed random zero of~$f_i$.
  Concerning the number of continuation steps~$K$,
  Theorem~\ref{thm:complexity-uncouple} gives (with $C'=5$ given by Lemma~\ref{lem:gamma-frob-lip})
  \begin{align*}
    \bE[K] &\leq 9000  \, n^3 \bigg( \sum_{i=1}^n \bE \left[\gamma_\Frob(f_i, \zeta_i)^2\right] \bigg )^\frac12 \\
           &\leq 9000   \,n^3 \bigg( \sum_{i=1}^n \tfrac14 d_i^3 (d_i+n) \bigg )^\frac12, && \text{by Lemma~\ref{lem:expect-gamma-frob},} \\
           &\leq 9000  \,n^3 \big(\tfrac14 n D^3(D+n) \big)^\frac12 \\
           &\leq 9000  \,n^4 D^2, && \text{with~$D+n \leq 2Dn$.}
  \end{align*}
  
  Concerning the total number of operations,
  the cost of Algorithm ``Solve'' splits into the cost of the sampling 
  and the cost of the numerical continuation.
  The former is $O(n^4 + nD^4)$ on average, by Proposition~\ref{prop:complexity-sampling-gaussian}.
  As for the latter, 
  the cost of a step is~$O(n^2D^2N)$
  (Corollary~\ref{coro:complexity-splitgamma}),
  and
  we need $O(n^4 D^2)$ continuation steps on
  average; as~$N\to \infty$, that is~$O(n^6 D^4 N)$ operations, by Corollary~\ref{coro:complexity-splitgamma}.
  When~$\min(n,D)\to \infty$, then both~$n$ and~$D$ are~$\binom{n+D}{n}^{o(1)} =
  N^{o(1)}$.
\end{proof}

\raggedright
\printbibliography

\end{document}